\documentclass[12pt,leqno]{article}

\usepackage{graphicx}

\usepackage{palatino}
\usepackage[mathcal]{euler}

\usepackage{amsmath,amsthm,amsfonts,amssymb,latexsym,amscd,enumerate}

\usepackage{fullpage}

\usepackage{xy}
\xyoption{all}

\theoremstyle{plain}
    \newtheorem{theorem}{Theorem}[section]
    \newtheorem{lemma}[theorem]{Lemma}
    \newtheorem{corollary}[theorem]{Corollary}
    \newtheorem{proposition}[theorem]{Proposition}

 \theoremstyle{definition}
    \newtheorem{definition}[theorem]{Definition}
    
    \newtheorem{example}[theorem]{Example}

    \newtheorem{remark}[theorem]{Remark}

\theoremstyle{remark}

\numberwithin{equation}{section}

    \newcommand{\R}{\mathbb{R}}
    \newcommand{\C}{\mathbb{C}} 
    \newcommand{\N}{\mathbb{N}}
    \newcommand{\Z}{\mathbb{Z}}

\newcommand{\kg}{\mathfrak{g}} 
\newcommand{\kk}{\mathfrak{k}}  
\newcommand{\kt}{\mathfrak{t}}
\newcommand{\kh}{\mathfrak{h}} 
\newcommand{\kp}{\mathfrak{p}} 

\newcommand{\cS}{\mathcal{S}}
\newcommand{\cL}{\mathcal{L}}
\newcommand{\cH}{\mathcal{H}}

\newcommand{\cF}{\mathcal{F}}

\newcommand{\cN}{\mathcal{N}}

\newcommand{\ddt}{\left. \frac{d}{dt}\right|_{t=0}}

\newcommand{\nabL}{\nabla}
\newcommand{\vnab}{v^{\nabL}}
\newcommand{\munab}{\mu^{\nabL}}
\newcommand{\munabinv}{(\munab)^{-1}}
\newcommand{\munabN}{\mu^{\nabla^N}}
\newcommand{\munabNinv}{(\munabN)^{-1}}
\newcommand{\munabM}{\mu^{\nabla^M}}
\newcommand{\munabMinv}{(\munabM)^{-1}}

\newcommand{\tilDpt}{\widetilde{D}_{p, t}}

\newcommand{\gstarM}{\kg^*_M}

\DeclareMathOperator{\SO}{SO}

\DeclareMathOperator{\Spin}{Spin}
\DeclareMathOperator{\U}{U}
\DeclareMathOperator{\SU}{SU}
\newcommand{\Spinc}{\Spin^c}

\DeclareMathOperator{\UF}{UF}
\DeclareMathOperator{\Real}{Re}

\DeclareMathOperator{\Stab}{Stab}
\DeclareMathOperator{\Destab}{Destab}

\newcommand{\Sj}{\sum_{j=1}^{d_G}}
\newcommand{\Sk}{\sum_{k=1}^{d_M}}

\DeclareMathOperator{\Ad}{Ad}
\DeclareMathOperator{\ad}{ad}
\DeclareMathOperator{\End}{End}
\DeclareMathOperator{\id}{id}

\DeclareMathOperator{\Crit}{Crit}
\DeclareMathOperator{\Cl}{Cl}

\DeclareMathOperator{\DInd}{D-Ind}
\DeclareMathOperator{\KInd}{K-Ind}

\DeclareMathOperator{\Gind}{\text{$G$}-index}
\DeclareMathOperator{\Kind}{\text{$K$}-index}
\DeclareMathOperator{\ind}{index}

\DeclareMathOperator{\grad}{grad}
\DeclareMathOperator{\Res}{Res}

\DeclareMathOperator{\topp}{top}

\DeclareMathOperator{\ch}{ch}

\newcommand{\red}{r}

\begin{document}

\title{Quantising proper actions on $\Spin^c$-manifolds}  
\date{\today}
\author{Peter Hochs\footnote{
University of Adelaide, \texttt{peter.hochs@adelaide.edu.au}} 
\hspace{1mm} 
and Varghese Mathai\footnote{
University of Adelaide, \texttt{mathai.varghese@adelaide.edu.au}}}  

\maketitle
\thispagestyle{empty}

\begin{abstract}
Paradan and Vergne generalised the quantisation commutes with reduction principle of Guillemin and Sternberg from symplectic to $\Spinc$-manifolds. We extend their result to noncompact groups and manifolds. This leads to a result for cocompact actions, and a result for non-cocompact actions for reduction at zero. The result for cocompact actions is stated in terms of $K$-theory of group $C^*$-algebras, and the result for non-cocompact actions is an equality of numerical indices. In the non-cocompact case, the result generalises to $\Spinc$-Dirac operators twisted by vector bundles. This yields an index formula for Braverman's analytic index of such operators, in terms of characteristic classes on reduced spaces.
\end{abstract}

\tableofcontents


\section{Introduction}

Recently, Paradan and Vergne \cite{PV} generalised the \emph{quantisation commutes with reduction} principle \cite{GuSt82, Meinrenken, MS,  Paradan1, TZ98, VergneICM} from the symplectic setting to the $\Spinc$-setting. In this paper, we extend their result to noncompact groups and manifolds. Whereas Paradan and Vergne use topological methods, we generalise Tian and Zhang's analytic approach \cite{HM, TZ98} to possibly non-cocompact actions on $\Spinc$-manifolds. This approach generalises to $\Spinc$-Dirac operators twisted by vector bundles, and implies an index formula for Braverman's analytic index \cite{Braverman2} of such operators.
For cocompact actions, we generalise and apply the $KK$-theoretic \emph{quantisation commutes with induction} methods of \cite{HochsDS, HochsPS}. Applications of our results include a proof of a $\Spinc$-version of Landsman's conjecture \cite{Landsman}, and various topological properties of the index of twisted $\Spinc$-Dirac operators for possibly non-cocompact actions. 


\subsection*{The compact case}

Cannas da Silva, Karshon and Tolman noted in \cite{CdSKT} that $\Spinc$-quantisation is the most general, and possibly natural, notion of geometric quantisation. 
This version of quantisation has a much greater scope for applications than geometric quantisation in the symplectic setting.
It was shown in Theorem 3 of \cite{CdSKT} that $\Spinc$-quantisation commutes with reduction for circle actions on compact $\Spinc$-manifolds, under a certain assumption on the fixed points of the action. Paradan and Vergne's result generalises this to actions by arbitrary compact, connected Lie groups, without the additional assumption made in \cite{CdSKT}.

Paradan and Vergne considered a compact, connected  Lie group $K$ acting on a compact, connected, even-dimensional manifold $M$, equipped with a $K$-equivariant $\Spinc$-structure. For a $\Spinc$-Dirac operator $D$ on $M$, they defined the $\Spinc$-quantisation of the action as
\[
Q_K^{\Spinc}(M) := \Kind(D), 
\]
which lies in the representation ring of $K$,
and computed the multiplicities $m_{\pi}$ in 
\[
Q_K^{\Spinc}(M) = \bigoplus_{\pi \in \hat K} m_{\pi} \pi.
\]
These multiplicities are expressed in terms of indices of $\Spinc$-Dirac operators on reduced spaces
\[
M_{\xi} := \mu^{-1}\bigl(\Ad^*(K)\xi \bigr)/K,
\]
where $\xi \in \kk^*$, and the \emph{$\Spinc$-momentum map} $\mu: M \to \kk^*$ is a generalisation of the momentum map in symplectic geometry.

\subsection*{The cocompact case}

We first generalise this result to cocompact actions by a Lie group $G$ on a manifold $M$, i.e.\ actions for which $M/G$ is compact. This is achieved by applying the quantisation commutes with induction machinery of \cite{HochsDS, HochsPS} to it, together with a $\Spinc$-slice theorem. In the cocompact case, one can define $\Spinc$-quantisation using the analytic assembly map, denoted by $\Gind$, from the Baum--Connes conjecture \cite{BCH}:
\[
Q_G^{\Spinc}(M) := \Gind(D) \quad \in K_*(C^*G),
\]
where $K_*(C^*G)$ is the $K$-theory of the maximal group $C^*$-algebra of $G$. This notion of quantisation was introduced by Landsman \cite{Landsman} in the symplectic setting. He conjectured that quantisation commutes with reduction at zero in that case.

To obtain a statement for reduction at nonzero values of the momentum map, we apply the natural map
\[
r_*: K_*(C^*G) \to K_*(C^*_rG),
\]
where $C^*_rG$ is the reduced $C^*$-algebra of $G$.
The group $K_*(C^*_rG)$ has natural generators $[\lambda]$, which have representation theoretic meaning in many cases. The first main result in this paper, Theorem \ref{thm [Q,R]=0 cocpt}, yields an expression for the multiplicities $m_{\lambda}$ in 
\begin{equation} \label{eq [Q,R]=0 cocpt intro}
r_*\bigl(Q_G^{\Spinc}(M) \bigr) = \sum_{\lambda} m_{\lambda} [\lambda].
\end{equation}

\subsection*{The non-cocompact case}

In the symplectic setting, the invariant part  of geometric quantisation was defined in \cite{HM} for possibly non-cocompact actions. Braverman \cite{Braverman2} then combined techniques from \cite{Braverman} and \cite{HM} to extend this definition to general Dirac operators, and proved important properties of the resulting index. We  generalise the main result from \cite{HM} from symplectic to $\Spinc$-manifolds. In addition, we obtain a generalisation to $\Spinc$-Dirac operators twisted by arbitrary vector bundles $E \to M$. This allows us to express Braverman's index of such operators in terms of topological data on $M_0$.

%
%

To be more precise, let $D_p^E$ be the $\Spinc$-Dirac operator on $M$, twisted by $E$ via a connection on $E$, for the $\Spinc$-structure whose determinant line bundle is the $p$'th tensor power of the determinant line bundle of a fixed $\Spinc$-structure. Then in Theorem \ref{thm index}, we obtain the index formula
\begin{equation} \label{eq index intro}
\ind^G D^E_{p} =  \int_{M_0} \ch(E_0) e^{\frac{p}{2}c_1(L_0)} \hat A(M_0),
\end{equation}
for $p \in \N$ large enough, where $\ind^G$ denotes Braverman's index \cite{Braverman2}. Here $E_0 := (E|_{\munabinv(0)})/G$, and $L_0 := (L|_{\munabinv(0)})/G$. This equality holds if $M_0$ is smooth, and a generalisation of the Kirwan vector field has a cocompact set of zeros. This implies that $M_0$ is compact. If $M$ and $G$ are both compact, analogous results were obtained in \cite{PV3, TZ98b}.

If  $M/G$ is noncompact, it is not clear a priori how to define a topological couterpart to Braverman's index. Gromov and Lawson \cite{GL83} face a similar problem in their study of Dirac operators on noncompact manifolds. They define a \emph{relative} topological index, representing the \emph{difference} of the indices of two operators satisfying their criteria, although these indices are not defined for each operator separately. They prove that the relative topological index equals the difference of the analytical indices of the operators in question (Theorem 4.18 in \cite{GL83}). Localisation to $\munabinv(0)$ allows us to give the topological expression \eqref{eq index intro} for the index of a single twisted $\Spinc$-Dirac operator, i.e.\ an `absolute' rather than a relative index formula.

\subsection*{Applications and examples}

If $M/G$ is compact, Theorem \ref{thm [Q,R]=0 large p} implies that the main result of \cite{MZ}, which to a large extent solves
Landsman's conjecture mentioned above,  generalises to the $\Spinc$-setting. 
We give a way to construct examples where our results apply, from cases where the group acting is compact. The main result \eqref{eq [Q,R]=0 cocpt intro} in the cocompact case has a purely geometric consequence, not involving $K$-theory and $C^*$-algebras. A special case of this consequence is an expression for the formal degree of a discrete series representation in terms of an $\hat A$-type genus of the corresponding coadjoint orbit. Finally, the index formula \eqref{eq index intro} allows us to draw conclusions about topological properties of the index of twisted $\Spinc$-Dirac operators. These include an excision property, and a twisted version of Hirzebruch's signature theorem in the noncompact case.

\subsection*{Outline of this paper}

In Section \ref{sec Dirac red}, we first briefly recall the definition of  $\Spinc$-Dirac operators. Then we state the definition of $\Spinc$-reduction as in \cite{PV}, and define stabilisation and destabilisation of $\Spinc$-structures in terms of Plymen's two out of three lemma. We give conditions for reduced spaces to have naturally defined $\Spinc$-structures in Section \ref{sec Spinc red slice}. We also discuss a $\Spinc$-slice theorem, and its relation to $\Spinc$-reduction.

Section \ref{sec result cocpt} contains the statements of Paradan and Vergne's result from \cite{PV}, and our  main result on cocompact actions, Theorem \ref{thm [Q,R]=0 cocpt}. This result is proved in Section \ref{sec cocpt}. 

The main result for untwisted $\Spinc$-Dirac operators for possibly non-cocompact actions, Theorem \ref{thm [Q,R]=0 large p}, is stated in Section \ref{sec non-cocpt case}. It is proved in Sections \ref{sec Bochner} and \ref{sec loc est}. The index formula for $\Spinc$-Dirac operators twisted by vector bundles is also stated in Section \ref{sec non-cocpt case}, and is proved in Section \ref{sec twisted}.

Finally, in Section \ref{sec appl ex}, we mention some applications  of the main results, and a way to construct examples where they apply.

\subsection*{Acknowledgements}

The authors are grateful to Paul-\'Emile Paradan and Mich\`ele Vergne, for useful comments and explanations, and for making a preliminary version of their paper \cite{PV2} available to them. They would also like to thank Gennadi Kasparov for a helpful remark.

The first author was supported by Marie Curie fellowship PIOF-GA-2011-299300 from the European Union. The second author thanks the Australian Research Council for support via the ARC Discovery Project grant DP130103924.

\subsection*{Notation and conventions}

We will denote the dimension of a manifold $Y$ by $d_Y$. If a group $H$ acts on $Y$, we denote the quotient map $Y \to Y/H$ by $q$, or by $q_H$ to emphasise which group is acting. For a finite-dimensional representation space $V$ of $H$, we write $V_Y$ for the trivial vector bundle $M \times V \to M$, with the diagonal $H$-action. (So that, for proper, free actions, $V_Y/H \to Y/H$ is the vector bundle associated to the principal fibre bundle $Y \to Y/H$.)
If $E\to Y$ is a real vector bundle of rank $r$, we will refer to a principal $\Spinc(r)$-bundle $P_E \to Y$ such that $E\cong P_E \times_{\Spinc(r)}\R^r$ as a $\Spinc$-structure on $E$, without making explicit mention of this isomorphism. 

\part{Preliminaries}


\section{Dirac operators and reduced spaces} \label{sec Dirac red}

Let $G$ be a Lie group, acting properly on a manifold $M$. Suppose $M$ is equipped with a $G$-equivariant $\Spinc$-structure. Let $L \to M$ be the associated determinant line bundle, and let a $G$-invariant, Hermitian connection $\nabL$ on $L$ be given. To these data, one can associate a $\Spin^c$-Dirac operator on $M$ in the usual way, as well as a \emph{$\Spin^c$-momentum map}, as introduced by Paradan and Vergne \cite{PV}. This momentum map can be used to define reduced spaces, which play a central role in the results in this paper. We mention Plymen's two out of three lemma, which we will use to construct $\Spinc$-structures on these reduced spaces in Section \ref{sec Spinc red slice}.

\subsection{Dirac operators} \label{sec Dirac ops}

Let $\cS \to M$ be the spinor bundle associated to the $\Spinc$-structure on $M$. The connection $\nabL$ and the Levi--Civita connection on $TM$ (associated to the Riemannian metric induced by the $\Spinc$-structure), together induce a connection $\nabla^{\cS}$ on $\cS$, as discussed for example in Proposition D.11 in \cite{LM}. The construction of the connection $\nabla^{\cS}$ involves local decompositions 
\[
\cS|_U \cong \cS^{\Spin}_U \otimes L|_U^{1/2}
\]
on open sets $U \subset M$, where $\cS^{\Spin}_U$ is the spinor bundle associated to a local $\Spin$-structure, to which the Levi--Civita connection lifts.

Let 
\[
c: TM \to \End(\cS)
\]
be the Clifford action. Identifying $T^*M \cong TM$ via the Riemannian metric,  one gets an action
\[
c: T^*M \otimes \cS \to \cS.
\]
The \emph{$\Spinc$-Dirac operator} associated to the $\Spinc$-structure on $M$ and the connection $\nabL$ on $L$ is then defined as the composition
\[
D: \Gamma^{\infty}(\cS) \xrightarrow{\nabla^{\cS}} \Omega^1(M; \cS) \xrightarrow{c} \Gamma^{\infty}(\cS).
\]
Write $d_M := \dim(M)$. If $\{e_1, \ldots, e_{d_M}\}$ is a local orthonormal frame for $TM$, then, locally, 
\[
D = \sum_{j=1}^{d_M}  c(e_j) \nabla^{\cS}_{e_j}.
\]

For certain arguments, we will also need the operator $D_p$ on the vector bundle $\cS_p := \cS \otimes L^p$, defined in the same way by a connection on $\cS_p$ which is induced by the Levi--Civita connection and $\nabL$, via local decompositions
\begin{equation} \label{eq decomp Sp}
\cS_p|_U \cong \cS^{\Spin}_U \otimes L|_U^{p+ 1/2}.
\end{equation}
Note that $\cS_p$ is the spinor bundle of the $\Spinc$-structure on $M$ obtained by twisting the original $\Spinc$-structure by the line bundle $L^p$ (see e.g.\ (D.15) in \cite{LM}).


\subsection{Momentum maps and reduction} \label{sec mom maps}

A $\Spinc$-momentum map is a generalisation of the momentum map in symplectic geometry. It was used by Paradan and Vergne in \cite{PV}. (See also Definition 7.5 in \cite{BGV}.)

For $X \in \kg$, let $X^M$ be the induced vector field on $M$, and let $\cL^E_X$ be the Lie derivative of sections of any $G$-vector bundle $E \to M$.
\begin{definition}
The \emph{$\Spin^c$-momentum map} associated to the connection $\nabL$ is the map
\[
\munab: M \to \kg^*
\]
defined by\footnote{In \cite{PV}, a factor $-i/2$ is used instead of $2\pi i$. Our convention is consistent with \cite{HM, TZ98}.}
\begin{equation} \label{eq def Spinc mom map}
2\pi i \munab_X  = \nabL_{X^M} - \cL^L_{X} \quad \in \End(L) = C^{\infty}(M),
\end{equation}
for any $X \in \kg$. Here $\munab_X$ denotes the pairing of $\munab$ with $X$.
\end{definition}

The notion of a $\Spin^c$-momentum map is a special case of the notion of an \emph{abstract moment map}, as for example in Definition 3.1 of \cite{GGK}. This is an equivariant map
\[
\Phi: M \to \kg^*
\]
such that for all $X \in \kg$, the pairing $\Phi_X$ of $\Phi$ with $X$ is locally constant on the set $\Crit(X^M)$ of zeros of the vector field $X^M$. A $\Spinc$-momentum map is an abstract moment map in this sense. This was already noted in the introduction to \cite{Paradan1}, and follows from the following well-known fact.
\begin{lemma} \label{lem dmuX}
For any G-equivariant line bundle $L \to M$ and a $G$-invariant connection $\nabL$ on $L$, and any $X \in \kg$, one has
\[
2\pi i d\munab_X = R^{\nabL}(\relbar, X^M),
\]
with $R^{\nabL}$ the curvature of $\nabL$. 
\end{lemma}
%
%
\begin{proof}
Let $u$ be any vector field on $M$. Then for all $X \in \kg$ and $s \in \Gamma^{\infty}(L)$, one has
\begin{equation} \label{eq abs mom 0}
2\pi i  \nabL_u(\munab_X s) = 2\pi i  u(\munab_X)s +  2\pi i  \munab_X \nabL_u s.
\end{equation}
This is also equal to
\begin{equation} \label{eq abs mom 1}
\nabL_u \bigl(   \nabL_{X^M} - \cL^L_{X} \bigr)s.
\end{equation}
Now 
\[
\nabL_u \nabL_{X^M}=\nabL_{X^M}\nabL_u+\nabL_{[u, X^M]} + R^{\nabL}(u, X^M).
\]
Also, by $G$-invariance of $\nabL$,
\[
\begin{split}
\nabL_u \cL^L_X s &= \ddt \nabL_u \exp(tX)s \\
	&= \ddt \exp(tX) \nabL_{\exp(-tX)^*u} s \\
	&= \cL^L_X \nabL_u s - \nabL_{[X^M,u]}s.
\end{split}
\]
We conclude that \eqref{eq abs mom 1} equals
\[
\nabL_{X^M}\nabL_u s+\nabL_{[u, X^M]} s + R^{\nabL}(u, X^M) s
- \cL^L_X \nabL_u s + \nabL_{[X^M,u]}s = 2\pi i \munab_X\nabL_u s + R^{\nabL}(u, X^M) s.
\]
Since this expression equals \eqref{eq abs mom 0}, we find that
\[
2\pi i u(\munab_X) = R^{\nabL}(u, X^M).
\]
\end{proof}


Analogously to symplectic reduction \cite{MW}, one can define reduced spaces in the $\Spinc$-setting.
\begin{definition}
For any $\xi \in \kg^*$, the space
\[
M_{\xi} := \munabinv(\xi)/G_{\xi} = \munabinv\bigl(\Ad^*(G)\xi \bigr)/G
\]
is the \emph{reduced space} at $\xi$.
\end{definition}

As in the symplectic case, the stabiliser $G_{\xi}$ acts infinitesimally freely on $\mu^{-1}(\xi)$, if $\xi$ is a regular value of $\munab$. Since $M_{\xi} \cong \munabinv(\xi)/{G_{\xi}}$,
this implies that the reduced space $M_{\xi}$ is an orbifold if $\xi$ is a regular value of $\munab$, and the action is proper. 
\begin{lemma} \label{lem loc free}
In the setting of Lemma \ref{lem dmuX}, let $\xi \in \kg^*$ be a regular value of $\munab$. Then for all $m \in \mu^{-1}(\xi)$, the infinitesimal stabiliser $\kg_m$ is zero.
\end{lemma}
\begin{proof}
In the situation of the lemma, let $X \in \kg_m$. Then for all $v \in T_mM$, we saw in Lemma \ref{lem dmuX} that
\[
\langle T_m\munab(v), X \rangle = v(\munab_X) (m) = \frac{1}{2\pi i}R^{\nabL}_m(v, X^M_m) = 0,
\]
since $X^M_m = 0$. Because $ T_m\munab$ is surjective, it follows that $X = 0$.
\end{proof}
(See Lemma 5.4 in \cite{GGK} for a version of this lemma where $G$ is a torus and $\munab$ is replaced by any abstract momentum map.)

%

\subsection{Stabilising and destabilising $\Spinc$-structures} \label{sec stab destab}

To study $\Spinc$-structures on reduced spaces, we will use the notions of \emph{stabilisation} and \emph{destabilisation} of $\Spinc$-structures. These will also be used to obtain a $\Spinc$-slice theorem in 
 Subsection \ref{sec fibred prod}.

Stabilisation and destabilisation are based on Plymen's \emph{two out of three lemma}.
\begin{lemma} \label{lem 2 out of 3}
Let $E, F \to M$ be oriented vector bundles with metrics, over a manifold $M$. Then $\Spinc$-structures on two of the three vector bundles $E$, $F$ and $E \oplus F$ determine a unique $\Spinc$-structure on the third. The determinant line bundles $L_E$, $L_F$ and $L_{E\oplus F}$ of the respective $\Spinc$-structures are related by
\[
L_{E\oplus F} = L_E \otimes L_F.
\]
\end{lemma}
\begin{proof}
See Section 3.1 of \cite{Plymen}. The uniqueness part of the statement refers to the constructions given there.
\end{proof}
\begin{remark} \label{rem G 2 out of 3}
Suppose a group $G$ acts on the vector bundles $E$ and $F$ in Lemma \ref{lem 2 out of 3}, and the two $\Spinc$-structures initially given in the lemma are $G$-equivariant. Then the $\Spinc$-structure on the third bundle, as constructed in Section 3.1 of \cite{Plymen}, is also $G$-equivariant. Here one uses the fact that the actions by $G$ on the spinor bundles associated to the $\Spinc$-structures on $E$, $F$ and $E\oplus F$ are compatible, since they are induced by the actions by $G$ on $E$ and $F$.
\end{remark}

\begin{definition} \label{def stab destab}
In the setting of Lemma \ref{lem 2 out of 3}, suppose $E$ and $F$ have $\Spinc$-structures. Let $P_E$ be the $\Spinc$-structure on $E$. Then the resulting $\Spinc$-structure on $E \oplus F$ is the \emph{stabilisation}
\[
\Stab_F(P_E) \to M.
\]

If $F$ and $E \oplus F$ have $\Spinc$-structures, and $P_{E\oplus F}$ is the $\Spinc$-structure on $E\oplus F$, then the resulting $\Spinc$-structure on $E$ is the \emph{destabilisation}
\[
\Destab_{F}(P_{E\oplus F}) \to M.
\]
\end{definition}
The terms stabilisation and destabilisation are motivated by the case where $F$ is a trivial vector bundle. See also Section 3.2 in \cite{Plymen}, Lemma 2.4 in \cite{CdSKT} and Section D.3.2 in \cite{GGK}.

We will use the following properties of stabilisation and destabilisation of $\Spinc$-structures.
\begin{lemma} \label{lem destab stab}
Let $E, F \to M$ be vector bundles with $\Spinc$-structures over a manifold $M$. Then
\begin{align}
\Stab_E \circ \Destab_E &= \id; \label{eq stab destab} \\
\Destab_E \circ \Stab_E &= \id; \label{eq destab stab} \\
\Stab_E \circ \Stab_F &= \Stab_{E \oplus F}; \label{eq stab stab} \\
\Destab_E \circ \Destab_F &= \Destab_{E \oplus F}. \label{eq destab destab}
\end{align}
(Here $\id$ means leaving $\Spinc$-structures on the relevant bundles unchanged.)
\end{lemma}
\begin{proof}
The relations \eqref{eq stab destab} and \eqref{eq destab stab} follow from the uniqueness part of Lemma \ref{lem 2 out of 3}. The explicit constructions in Section 3.1 of \cite{Plymen} imply that \eqref{eq stab stab} and \eqref{eq destab destab} hold. 
\end{proof}


\section{$\Spinc$-structures on reduced spaces} \label{sec Spinc red slice}

Consider the setting of Subsection \ref{sec mom maps}. One can define quantisation of smooth or orbifold reduced spaces using $\Spinc$-structures induced by the $\Spinc$-structure on $M$. If $G$ is a torus, these are described in Proposition D.60 of \cite{GGK}. In general, we will see that the $\Spinc$-structure on $M$ induces one on reduced spaces at \emph{$\Spinc$-regular values} of the $\Spinc$-momentum map $\munab$. In Proposition \ref{prop Spinc reg}, we give a relation between $\Spinc$-regular values and usual regular values. We then discuss how Abels' slice theorem for proper actions can be used in the $\Spinc$-context, and how it is related to $\Spinc$-reduction. The proofs of the main statements in this section will be given Section \ref{sec cocpt}.

\subsection{$\Spinc$-regular values} \label{sec Spinc red}

For $\xi \in \kg^*$, we will denote the quotient map $\munabinv(\xi) \to M_{\xi}$  by $q$.
\begin{definition} \label{def Spinc reg}
A value $\xi \in \munab(M)$ of $\munab$ is a \emph{$\Spinc$-regular value}  if
\begin{itemize}
\item $\munabinv(\xi)$ is smooth;
\item $G_{\xi}$ acts locally freely on $\munabinv(\xi)$; and
\item there is a $G_{\xi}$-invariant splitting
\[
TM|_{(\mu^{\nabla})^{-1}(\xi)} = q^*TM_{\xi} \oplus \cN^{\xi},
\]
for a vector
bundle $\cN^{\xi} \to (\mu^{\nabla})^{-1}(\xi)$ with a $G_{\xi}$-equivariant $\Spin$-structure.
\end{itemize}
\end{definition}
\begin{remark}
The third point in Definition \ref{def Spinc reg} appears to have a choice of the bundle $\cN^{\xi}$ in it, but these are all isomorphic; the condition is really that the quotient bundle
\[
TM|_{(\mu^{\nabla})^{-1}(\xi)} / q^*TM_{\xi} 
\]
has a $G_{\xi}$-equivariant $\Spin$-structure. 

Note that a $\Spin$-structure  is equivalent to a $\Spinc$-structure with a trivial determinant line bundle. In the equivariant setting, an equivariant $\Spin$-structure is equivalent to a $\Spinc$-structure with an equivariantly  trivial determinant line bundle. Indeed, 
if  the determinant line bundle of a $\Spinc$-structure  is equivariantly  trivial, then its
 spinor bundle equals the spinor bundle of the underlying $\Spin$-structure as equivariant vector bundles.
\end{remark}

\begin{lemma} \label{lem Spinc normal}
If $\xi$ is a $\Spinc$-regular value of $\munab$,
then the $\Spinc$-structure on $M$ induces an orbifold $\Spinc$-structure on $M_{\xi}$, with determinant line bundle
\[
L_{\xi} := (L|_{(\mu^{\nabla})^{-1}(\xi)})/G_{\xi} \to M_{\xi}
\]
\end{lemma}
\begin{proof}
We generalise the proof of Proposition D.60 in \cite{GGK} to cases where $G$ may not be a torus.

We apply the equivariant version of Lemma \ref{lem 2 out of 3} (see Remark \ref{rem G 2 out of 3}) to the vector bundles $q^*TM_{\xi}$ and $\cN^{\xi}$. This yields a $G_{\xi}$-equivariant $\Spinc$-structure on $q^*TM_{\xi}$, with determinant line bundle $L|_{(\mu^{\nabla})^{-1}(\xi)}$. On the quotient $M_{\xi}$, this induces an orbifold $\Spinc$-structure, with determinant line bundle $L_{\xi}$.
\end{proof}
\begin{remark}
If $G_{\xi}$ acts \emph{freely} on $(\mu^{\nabla})^{-1}(\xi)$, then one can also use the $\Spinc$-structure
\begin{equation} \label{eq Spinc red free}
(P_M|_{(\mu^{\nabla})^{-1}(\xi)})/G_{\xi} \to M_{\xi},
\end{equation}
 on $(TM|_{(\mu^{\nabla})^{-1}(\xi)})/G_{\xi}$,
where $P_M \to M$ is the given $\Spinc$-structure on $M$. 
The determinant line bundle of \eqref{eq Spinc red free} is
 $L_{\xi}$. By the assumption on $\cN^{\xi}$, Lemma \ref{lem 2 out of 3} yields a $\Spinc$-structure on $TM_{\xi}$, with the same determinant line bundle. 
 
 If $G_{\xi}$ only acts locally freely on $(\mu^{\nabla})^{-1}(\xi)$, then one would need an orbifold version of Lemma \ref{lem 2 out of 3} to use this argument.
\end{remark}

In the language of Definition \ref{def stab destab}, the $\Spinc$-structure $P_{M_{\xi}}$ on $M_{\xi}$ induced by the  $\Spinc$-structure $P_{M}$ on $M$ equals
\begin{equation} \label{eq PMxi}
P_{M_{\xi}}  = \Destab_{\cN^{\xi}}\bigl(P_M|_{(\mu^{\nabla})^{-1}(\xi)} \bigr)/G_{\xi}.
\end{equation}

In Definition \ref{def Spinc reg}, it was not assumed that $\xi$ is a regular value of $\munab$ in the usual sense, since this will not necessarily be the case in the situation considered in Subsection \ref{sec fibred prod}. If $\xi$ is a regular value, then the first two conditions of Definition \ref{def Spinc reg} hold by Lemma \ref{lem loc free}. One can use the following fact to check the third condition.
\begin{proposition} \label{prop Spinc reg}
Suppose that $\xi$ is a regular value of $\munab$, and that
\begin{itemize}
\item $G$ and $G_{\xi}$ are unimodular;
\item there is an $\Ad(G_{\xi})$-invariant, nondegenerate bilinear form on $\kg$;
\item there is an $\Ad(G_{\xi})$-invariant subspace $V\subset \kg$ such that
\[
\kg = \kg_{\xi} \oplus V;
\]
and
\item there is an $\Ad(G_{\xi})$-invariant complex structure on $V$.
\end{itemize}
Then $\xi$ is a $\Spinc$-regular value of $\munab$.
\end{proposition}
\begin{example} \label{ex reg zero}
If $\kg_{\xi} = \kg$, then the last two conditions in Proposition  \ref{prop Spinc reg} are vacuous. Therefore,
\begin{itemize}
\item if $G$ is \emph{Abelian}, any regular value of $\munab$ is a $\Spinc$-regular value;
\item if $0$ is a regular value of $\munab$, and $G$ is semisimple, then $0$ is a $\Spinc$-regular value.
\end{itemize}
\end{example}
\begin{example} \label{ex reg se}
If $G$ is unimodular, and $G_{\xi}$ is \emph{compact} (i.e.\ $\xi$ is strongly elliptic), then one can use an $\Ad(G_{\xi})$-invariant inner product on $\kg$. Together with the standard symplectic form on 
\[
V:= \kg_{\xi}^{\perp} \cong \kg/\kg_{\xi} \cong T_{\xi}(G\cdot \xi),
\]
this induces an $\Ad(G_{\xi})$-invariant complex structure on $V$ (see e.g.\ Example D.12 in \cite{GGK}).

For semisimple Lie groups,  strongly elliptic elements and coadjoint orbits correspond to discrete series representations, under  an integrality condition. (See also \cite{Paradan2}.) 
\end{example}
\begin{remark}
If the bilinear form in the second point of Proposition \ref{prop Spinc reg} is positive definite on $\kg_{\xi}$, then one can take $V = \kg_{\xi}^{\perp}$, and the third condition  in Proposition \ref{prop Spinc reg} holds.

If, on the other hand, the bilinear form  is positive definite on $V$, then one has an induced $\Ad(G_{\xi})$-invariant complex structure on $V$ (as in Example \ref{ex reg se}), so the fourth condition in Proposition \ref{prop Spinc reg} holds.
\end{remark}


We will prove Proposition \ref{prop Spinc reg} in Subsection \ref{sec red reg}.

\subsection{$\Spinc$-slices} \label{sec fibred prod}

Let $G$ be an almost connected Lie group, and let $K<G$ be a maximal compact subgroup. Let $M$ be any smooth manifold, on which $G$ acts properly. Then Abels showed (see p.\ 2 of \cite{Abels}) that there is a $K$-invariant submanifold (or \emph{slice}) $N \subset M$ such that the map $[g, n] \mapsto g\cdot n$ is a $G$-equivariant diffeomorphism
\[
 G\times_K N \cong M. 
\]
Explicitly, the left hand side is
the quotient of $G\times N$ by the $K$-action given by
\[
k\cdot (g, n) = (gk^{-1}, kn),
\]
for $k \in K$, $g \in G$ and $n \in N$.

Fix an $\Ad(K)$-invariant inner product on $\kg$, and let $\kp \subset \kg$ be the orthogonal complement to $\kk$. After replacing $G$ by a double cover if necessary, we may assume that $\Ad: K \to \SO(\kp)$ lifts to 
\begin{equation} \label{eq tilde Ad}
\widetilde{\Ad}: K \to \Spin(\kp). 
\end{equation}
Indeed, consider the diagram
\[
\xymatrix{
\widetilde{K} \ar[r]^-{\widetilde{\Ad}} \ar[d]_{\pi_K}& \Spin(\kp) \ar[d]_{\pi}^{2:1}\\
K \ar[r]^-{\Ad} & \SO(\kp),
}
\]
where
\[
\begin{split}
\widetilde{K} &:= \{ (k, a) \in K\times \Spin(\kp); \Ad(k) = \pi(a)\}; \\
\pi_K(k, a)&:= k; \\
\widetilde{\Ad}(k, a) &:= a,
\end{split}
\]
for $k \in K$ and $a \in \Spin(\kp)$. Then for all $k \in K$,
\[
\pi_K^{-1}(k) \cong \pi^{-1}(\Ad(k)) \cong \Z_2,
\]
so $\pi_K$ is a double covering map. In what follows, we will assume the lift \eqref{eq tilde Ad} exists.

It was shown in Section 3.2 of \cite{HochsDS} that a $K$-equivariant $\Spin^c$-structure $P_N$ on $N$ induces a $G$-equivariant $\Spin^c$-structure $P_M$ on $M$. In terms of stabilisation of $\Spinc$-structures (Definition \ref{def stab destab}), one has 
\begin{equation} \label{eq spinc ind}
P_M = G\times_K \Stab_{\kp_N}(P_N).
\end{equation}
Here  $\kp_N \to N$ is the trivial vector bundle $N \times \kp \to N$, equipped with the $K$-action
\[
k(n, X) = (kn, \Ad(k)X),
\]
for $k \in K$, $n \in N$ and $X \in \kp$. It has the $K$-equivariant $\Spin$-structure
\begin{equation} \label{eq Spin pN}
N\times \Spin(\kp) \to N,
\end{equation}
with the diagonal $K$-action defined via the lift \eqref{eq tilde Ad} of the adjoint action.
To show that \eqref{eq spinc ind} defines a $\Spinc$-structure on $M$, one uses the isomorphism
\begin{equation} \label{eq decomp TM 1}
TM = G\times_K(TN \oplus \kp_N)
\end{equation}
(see Proposition 2.1 and Lemma 2.2 in \cite{HochsDS}).

Analogously to Section 2.4 in \cite{HochsDS} in the symplectic setting, the construction \eqref{eq spinc ind} is invertible. Indeed, given a G-equivariant $\Spinc$-structure $P_M \to M$ on $M$, consider the $K$-equivariant $\Spinc$-structure
\begin{equation} \label{eq restr Spinc}
P_N := \Destab_{\kp_N}(P_M|_N) \to N
\end{equation}
on $N$. Here 
we again use \eqref{eq decomp TM 1}. 
\begin{lemma} \label{lem ind restr spinc}
The constructions \eqref{eq spinc ind} and \eqref{eq restr Spinc} are inverse to one another.
\end{lemma}
\begin{proof}
Starting with a $K$-equivariant $\Spinc$-structure $P_N \to N$ on $N$, we see that \eqref{eq destab stab} implies that
\[
\Destab_{\kp_N}\bigl(    \bigl(G\times_K \Stab_{\kp_N}(P_N) \bigr)    |_N \bigr) = \Destab_{\kp_N}\bigl(   \Stab_{\kp_N}(P_N) \bigr) = P_N.    
\]
On the other hand, suppose $P_M \to M$ is a $G$-equivariant $\Spinc$-structure on $M$. Then we have by \eqref{eq stab destab},
\[
G\times_K \Stab_{\kp_N} \bigl(  \Destab_{\kp_N}(P_M|_N)   \bigr) = G\times_K(P_M|_N), 
\]
which is ismorphic to $P_M$
via the map
$[g, f] \mapsto g\cdot f$, for $g \in G$ and $f \in P_M|_N$.
\end{proof}

Combining Abels' theorem and Lemma \ref{lem ind restr spinc}, we obtain the following $\Spinc$-slice theorem.
\begin{proposition} \label{prop slice spinc}
For any G-equivariant $\Spinc$-structure $P_M$ on a proper $G$-manifold $M$, there is a $K$-invariant submanifold $N \subset M$ and a $K$-equivariant $\Spinc$-structure $P_N \to N$ such that $M \cong G\times_K N$, and
\[
P_M = G\times_K \Stab_{\kp_N}(P_N).
\]
\end{proposition}

\subsection{Reduction and slices} \label{sec slice red}

Consider the situation of Subsection \ref{sec fibred prod}, and fix $N$ and $P_N$ as in Proposition \ref{prop slice spinc}.
To relate $\Spinc$-reductions of the actions by $G$ on $M$ and by $K$ on $N$, we will use a relation between $\Spinc$-momentum maps for these two actions. Let $L^M \to M$ and $L^N \to N$ be the determinant line bundles of $P_M$ and $P_N$, respectively.
Let $\nabla^M$ be a $G$-invariant Hermitian connection on $L^M$, let $j: N\hookrightarrow M$ be the inclusion map, and consider the connection $\nabla^N := j^*\nabla^M$ on $L^N$. Let $\mu^{\nabla^M}: M \to \kg^*$ and $\mu^{\nabla^N}: N \to \kk^*$ be the $\Spinc$-momentum maps associated to these connections. Let $\Res^{\kg}_{\kk}: \kg^* \to \kk^*$ be the restriction map.
\begin{lemma} \label{lem L N M}
One has
\begin{enumerate}
\item $L^N = L^M|_N$;
\item $L^M = G\times_K L^N$;
\item $\mu^{\nabla^N} = \Res^{\kg}_{\kk} \circ \mu^{\nabla^M}|_N$;
\item if $\mu^{\nabla^M}(n) \in \kk^*$ for all $n \in N$, then 
\begin{equation} \label{eq mu N M 1}
\mu^{\nabla^M}([g, n]) = \Ad^*(g)\mu^{\nabla^N}(n),
\end{equation}
for all $g \in G$ and $n \in N$.
\end{enumerate}
\end{lemma}
In the fourth point of this lemma, and in the rest of this paper, we embed $\kk^*$ into $\kg^*$ as the annihilator of $\kp$.
\begin{proof}
The $\Spin$-structure \eqref{eq Spin pN} on $\kp_N$ induces a $\Spinc$-structure with equivariantly trivial determinant line bundle
 $L^{\kp_N} \to N$. Since
\[
P_N = \Destab_{\kp_N}(P_M|_N),
\]
Lemma \ref{lem 2 out of 3} implies that
\[
L^N = L^N \otimes L^{\kp_N} = L^M|_N.
\]
So the first claim holds, and the second claim follows from this: $L^M = G\cdot L^M|_N = G\times_K L^N$.

To prove the third claim, we use the first claim, and note that for all $X \in \kk$,
\[
2\pi i \mu^{\nabla^N}_{X} 
= \nabla^N_{X^N} - \cL^{L^N}_X 
 = \left. \left( \nabla^M_{X^M} - \cL^{L^M}_X \right) \right|_{\Gamma^{\infty}(L^N)} 
 = 2\pi i \mu^{\nabla^M}_X|_N.
\]
The fourth claim follows from the third.
\end{proof}
In the symplectic case, it was shown in Proposition 2.8 of \cite{HochsDS} that one may take $N = (\mu^{\nabla^M})^{-1}(\kk^*)$. Then the condition in the fourth point of Lemma \ref{lem L N M} holds, so one has \eqref{eq mu N M 1}.  In the $\Spinc$-setting, we use an arbitrary slice $N$. In Subsection \ref{sec ind nabla}, we show that a $K$-invariant connection $\nabla^N$ on $L^N$ induces a $G$-invariant connection $\nabla^M$ on $L^M$ such that the condition  in the fourth point of Lemma \ref{lem L N M} is satisfied (see Lemma \ref{lem mu N M}). From now on, we suppose that $\nabla^M$ was chosen in this way, so that \eqref{eq mu N M 1} holds.

In that case, a regular value of $\munabN$ is not necessarily a regular value of $\munabM$. Indeed, any tangent vector to $M$ at $[e,n]$, for $n \in N$, is of the form $T_{(e,n)}q(X, v) = X^M_{[e,n]}+v$, for $X \in \kg$ and $v \in T_nN$. Using \eqref{eq mu N M 1} one computes that
\[
T_{[e,n]}\munabM(X^M_{[e,n]}+v) = \ad^*(X)\bigl(\munabN(n) \bigr) 
	+ T_n\munabN(v).
\]
If, for example, $\munabN(n) = 0$, then $T_{[e,n]}\munabM$ can only be surjective if $\kg = \kk$, even if  $0$ is a regular value of $\munabN$. However, all regular values of $\munabN$ are $\Spinc$-regular values of $\munabM$.
\begin{proposition} \label{prop Spinc fibred}
If $\xi$ is a regular value of $\munabN$, then it is a $\Spinc$-regular value of $\munabM$.
\end{proposition}
Note that by the third point of Lemma \ref{lem L N M}, $\xi$ is a regular value of $\munabN$ if and only if it is a regular value of $\Res^{\kg}_{\kk} \circ \mu^{\nabla^M}$. Fix $\xi \in \kk^*$ satisfying this condition, and let $P_{M_{\xi}} \to M_{\xi}$ be the $\Spinc$-structure on $M_{\xi}$ as in Lemma \ref{lem Spinc normal}.

There is another way to define a $\Spinc$-structure on $M_{\xi}$, using the following fact.
\begin{lemma} \label{lem Nred Mred}
For any $\eta \in \kk^*$, the inclusion map map $N \hookrightarrow M$ induces a homeomorphism
\[
N_{\eta} \cong M_{\eta}.
\]
\end{lemma}
Since $\xi$ is a regular value of $\mu^{\nabla^N}$, Proposition \ref{prop Spinc reg} implies that the $\Spinc$-structure on $N$ induces a $\Spinc$-structure on $N_{\xi}$, which equals $M_{\xi}$. In the proof of Theorem \ref{thm [Q,R]=0 cocpt}, we will use the fact that the two $\Spinc$ structures $P_{M_{\xi}}$ and $P_{N_{\xi}}$ are the same. 
\begin{proposition} \label{prop PNxi PMxi}
The $\Spinc$-structures $P_{M_{\xi}}$ and $P_{N_{\xi}}$ on $M_{\xi} \cong N_{\xi}$ are equal.
\end{proposition}
Lemma \ref{lem Nred Mred} and Propositions \ref{prop Spinc fibred} and \ref{prop PNxi PMxi} will be proved in Subsections \ref{sec ind nabla}--\ref{sec PNxi PMxi}.

We end this section by mentioning a compatibility property of stabilising and destabilising $\Spin$-structures with the fibred product construction that appears in the slice theorem. This property will be used in the proof of Proposition \ref{prop PNxi PMxi}. 
Suppose $H<G$ is any closed subgroup, acting on a manifold $N$, and let $E \to N$ be an $H$-vector bundle with an $H$-equivariant $\Spinc$-structure $P_E \to N$. Then $G\times_H P_E \to G\times_H N$ is a $G$-equivariant $\Spinc$-structure for the $G$-vector bundle $G\times_H E \to G\times_H N$ (see Lemma 3.7 in \cite{HochsDS}). In the proof of Proposition \ref{prop PNxi PMxi}, we will use the fact that this construction is compatible with stabilisation and destabilisation.
\begin{lemma} \label{lem stab fibred}
In the above setting, let $F \to N$ be another $H$-vector bundle.
\begin{enumerate}
\item If $P_F \to N$ is an $H$-equivariant $\Spinc$-structure on $P_F$, then 
\[
G\times_H\Stab_E(P_F) = \Stab_{G\times_H E}(G\times_H P_F).
\]
\item If $P_{E\oplus F} \to N$ is an $H$-equivariant $\Spinc$-structure on $P_{E\oplus F}$, then 
\[
G\times_H\Destab_E(P_{E\oplus F}) = \Destab_{G\times_H E}(G\times_H P_{E\oplus F}).
\]
\end{enumerate}
\end{lemma}
\begin{proof}
The first point follows from the explicit constructions in Section 3.1 of \cite{Plymen}. Here one uses the fact that the spinor bundle associated to $G\times_H P_{E}$ is $G\times_H \cS_E$, where $\cS_E \to N$ is the spinor bundle associated to $P_E$. This is compatible with the grading operators. 

The second point can be proved in a similar way, or deduced from the first point, by using the fact that destabilisation is the inverse of stabilisation, as in \eqref{eq stab destab} and \eqref{eq destab stab}. 
\end{proof}

\begin{remark}
We have only considered the principle $\Spinc(r)$-bundle part $P_E \to X$ of a $\Spinc$-structure on a vector bundle  $E \to X$ of rank $r$ over a manifold $X$, not the isomorphism
\[
P_E \times_{\Spinc(r)} \R^{r} \cong E.
\] 
If $E$ is the tangent bundle to $X$, then this isomorphism determines the Riemannian metric on $X$ induced by the $\Spinc$-structure. For cocompact actions, where we will apply the material in this subsection, the index of the $\Spinc$-Dirac operator
 is independent of this metric, however.
\end{remark}


\part{Cocompact actions} \label{part cocpt}


\section{The result on cocompact actions} \label{sec result cocpt}


The main result on cocompact actions is Theorem \ref{thm [Q,R]=0 cocpt}, which states that
 that $\Spinc$-quantisation commutes with reduction at $K$-theory generators.
In this section, we state Paradan and Vergne's result for compact groups and manifolds in \cite{PV}, and Theorem \ref{thm [Q,R]=0 cocpt} for cocompact actions.  We will deduce Theorem \ref{thm [Q,R]=0 cocpt} from Paradan and Vergne's result in Section \ref{sec cocpt}. 



We keep using the notation of Section \ref{sec Dirac red}.

\subsection{The compact case} \label{sec cpt}

First of all, we define $\Spinc$-quantisation of sufficiently regular reduced spaces, which will always be compact in the settings we consider. Let $\xi$ be a $\Spinc$-regular value of $\munab$. Then by  Lemma \ref{lem Spinc normal}, the reduced space $M_{\xi}$ is a $\Spinc$-orbifold. 
Suppose that $M_{\xi}$ is compact and even-dimensional. Let $D_{M_{\xi}}$ be the $\Spinc$-Dirac operator on $M_{\xi}$, defined with the connection on the determinant line bundle $L_{\xi} \to M_{\xi}$ induced by a given connection on the determinant line bundle $L \to M$. 
\begin{definition} \label{def QR}
The \emph{$\Spinc$-quantisation} of $M_{\xi}$ is the index of $D_{M_{\xi}}$:
\[
Q^{\Spinc}(M_{\xi}) := \ind(D_{M_{\xi}}) \quad \in \Z.
\]
\end{definition}

Now suppose that $G = K$ is compact and connected. Suppose that $M$ is  even-dimensional, and also compact and connected. Since $M$ is even-dimensional, the spinor bundle $\cS$ splits into  even and odd parts, sections of which are interchanged by the
 $\Spinc$-Dirac operator $D$. Because $M$ is compact, this Dirac operator has finite-dimensional kernel, and one can define
 \begin{equation} \label{eq def quant cpt}
 Q^{\Spin^c}_K(M) := \Kind(D) = [\ker D^+] - [\ker D^-] \quad \in R(K),
 \end{equation}
 where $D^{\pm}$ are the restrictions of $D$ to the even and odd parts of $\cS$, repectively, and $R(K)$ is the representation ring of $K$.

Let $T<K$ be a maximal torus, with Lie algebra $\kt \subset \kk$. Let $\kt^*_+ \subset \kt^*$ be a choice of (closed) positive Weyl chamber. Let $R$ be the set of roots of $(\kk_{\C}, \kt_{\C})$, and let $R^+$ be the set of positive roots with respect to $\kt^*_+$. Set
\[
\rho_K := \frac{1}{2}\sum_{\alpha \in R^+} \alpha.
\]

 Let $\cF$ be the set of relative interiors of faces of $\kt^*_+$. Then
\[
\kt^*_+ = \bigcup_{\sigma \in \cF} \sigma,
\]
a disjoint union. For $\sigma \in \cF$, let $\kk_{\sigma}$ be the infinitesimal stabiliser of a point in $\sigma$. Let $R_{\sigma}$ be the set of roots of $\bigl( (\kk_{\sigma})_{\C}, \kt_{\C}\bigr)$, and let $R_{\sigma}^+ := R_{\sigma} \cap R^+$. Set
\[
\rho_{\sigma} := \frac{1}{2}\sum_{\alpha \in R^+_{\sigma}} \alpha.
\]
Note that, if $\sigma$ is the interior of $\kt^*_+$, then $\rho_{\sigma} = 0$.

For any subalgebra $\kh \subset \kk$, let $(\kh)$ be its conjugacy class. Set
\[
\cH_{\kk} := \{ (\kk_{\xi}); \xi \in \kk\}.
\]
For $(\kh) \in \cH_{\kk}$, write
\[
\cF(\kh) := \{\sigma \in \cF; (\kk_{\sigma}) = (\kh)\}.
\]
Let $(\kk^M)$ be the conjugacy class of the generic (i.e.\ minimal) infinitesimal stabiliser $\kk^M$ of the action by $K$ on $M$. Note that by Lemma \ref{lem loc free}, one has $(\kk^M) = 0$ if $\munab$ has regular values.

Let $\Lambda_+ \subset i\kt^*$ be the set of dominant integral weights. 
In the $\Spinc$-setting, it is natural to parametrise the irreducible representations by their infinitesimal characters, rather than by their highest weights. For $\lambda \in \Lambda_+ + \rho_K$, let $\pi^K_{\lambda}$ be the irreducible representation of $K$ with infinitesimal character $\lambda$, i.e.\ with  highest weight $\lambda - \rho_K$. Then one has, for such $\lambda$,
\[
Q^{\Spinc}(K\cdot \lambda) = \pi_{\lambda}^K,
\]
see Lemma 2.1 in \cite{PV}.

Write
\[
 Q^{\Spin^c}_K(M) = \bigoplus_{\lambda \in \Lambda_+ + \rho_K} m_{\lambda} [\pi^K_{\lambda}],
\]
with $m_{\lambda} \in \Z$. Then Paradan and Vergne proved the following expression for $m_{\lambda}$ in terms of reduced spaces.
\begin{theorem}[\cite{PV}, Theorem 3.4] \label{thm [Q,R]=0 cpt}
Suppose $([\kk^M, \kk^M]) = ([\kh, \kh])$, for $(\kh) \in \cH_{\kk}$. Then
\begin{equation} \label{eq [Q,R]=0 cpt}
m_{\lambda} = \sum_{  \begin{array}{c} \vspace{-1.5mm} \scriptstyle{\sigma \in \cF(\kh) \text{ s.t.}}\\  \scriptstyle{\lambda - \rho_{\sigma} \in \sigma} \end{array}} Q^{\Spin^c}(M_{\lambda - \rho_{\sigma}}).
\end{equation}
\end{theorem}
Here the quantisation $Q^{\Spin^c}(M_{\lambda - \rho_{\sigma}})$ of the reduced space\footnote{for $\xi \in i\kk^*$, we write $M_{\xi}:= M_{\xi/i}$.} $M_{\lambda - \rho_{\sigma}}$ is defined in Section 4 of \cite{PV}, which includes cases where Lemma \ref{lem Spinc normal} does not apply, and reduced spaces are singular. 


If the generic stabiliser $\kk^M$ is \emph{Abelian}, Theorem \ref{thm [Q,R]=0 cpt} simplifies considerably. 
As noted above, 
 this occurs in particular if $\munab$ has a regular value.
\begin{corollary} \label{cor cpt Ab stab}
If $\kk^M$ is Abelian, then
\[
m_{\lambda} =  Q^{\Spin^c}(M_{\lambda} ).
\]
\end{corollary}
\begin{proof}
If one takes $\kh = \kt$ in Theorem \ref{thm [Q,R]=0 cpt}, then $\cF(\kh)$ only contains the interior of $\kt^*_+$. Hence $\rho_{\sigma} = 0$, for the single element $\sigma \in \cF(\kh)$.
\end{proof}
In particular, if $0$ is a regular value of $\mu^{\nabla}$, then the invariant part of the $\Spin^c$-quantisation of $M$ is
\begin{equation} \label{eq [Q,R]=0 cpt zero}
 Q^{\Spin^c}_K(M)^K = Q^{\Spin^c}(M_{\rho_K}),
\end{equation}
since $\pi^K_{\rho_K}$ is the trivial representation.


\subsection{The cocompact case} \label{sec cocpt zero}

Now suppose $M$ and $G$ may be noncompact, but $M/G$ is compact. Then Landsman \cite{HL, Landsman} defined geometric quantisation via the \emph{analytic assembly map} from the Baum--Connes conjecture  \cite{BCH}. This takes values in the $K$-theory of the maximal or reduced group $C^*$-algebra $C^*G$ or $C^*_{\red}G$ of $G$. Landsman's definition extends directly to the $\Spinc$ case.
\begin{definition} \label{def quant cocpt}
If $M/G$ is compact, the \emph{$\Spinc$-quantisation} of the action by $G$ on $M$ is
\begin{equation} \label{eq def quant cocpt}
Q^{\Spinc}_G(M) := \Gind(D) \quad \in K_*(C^*G),
\end{equation}
where $\Gind$ denotes the analytic assembly map.
\end{definition}
In this definition, the maximal $C^*$-algebra $C^*G$ of $G$ was used. By applying the map
\[
r_*: K_*(C^*G) \to K_*(C^*_{\red}G)
\]
induced by the natural map $r: C^*G \to C^*_{\red}G$, one obtains the \emph{reduced\footnote{Note that the word `reduced' and the map $r_*$ used here have nothing to do with reduction; this is just an unfortunate clash of terminology.} $\Spin^c$-quantisation}
\[
Q^{\Spinc}_G(M)_{\red} := r_* \bigl( Q^{\Spinc}_G(M) \bigr)  \quad \in K_*(C^*_{\red}G).
\]
(This is equal to \eqref{eq def quant cocpt}, if $\Gind$ denotes the assembly map for $C^*_{\red}G$, but we include the map $r_*$ to make the distinction clear.)
If $G$ is compact, then $K_*(C^*G)$ and $K_*(C^*_{\red}G)$ equal the representation ring $R(G)$ of $G$. Then
 the above definitions of $\Spinc$-quantisation and reduced $\Spinc$-quantisation both reduce to \eqref{eq def quant cpt}. 

Landsman used the reduction map
\[
R_0: K_*(C^*G) \to \Z
\]
induced on $K$-theory by the continuous map
\[
C^*G \to \C,
\]
which on $C_c(G) \subset C^*G$ is given by integration over $G$. 
If $G$ is compact, then $R_0: R(G) \to \Z$ is taking the multiplicity of the trivial representation. 
Landsman conjectured that
\begin{equation} \label{eq conj [Q,R]=0 cocpt zero}
R_0 \bigl( Q_G(M)  \bigr) = Q(M_0),
\end{equation}
in the symplectic case (if $M_0$ is smooth). Here quantisation is defined as in Definition \ref{def quant cocpt}, where $D$ is  a Dirac operator coupled to a prequantum line bundle. 

This conjecture was proved by Hochs and Landsman  \cite{HL} for a specific class of groups $G$, and by Mathai and Zhang \cite{MZ} for general $G$, where one may need to replace the prequantum line bundle  by a tensor power. As a special case of Theorem \ref{thm [Q,R]=0 large p}, we will obtain a generalisation to the $\Spinc$-setting of Mathai and Zhang's result on the Landsman conjecture  (see Corollary \ref{lem cocpt special case}). This asserts that \eqref{eq conj [Q,R]=0 cocpt zero} still holds for $\Spinc$-quantisation, for a well-chosen $\Spinc$-structure on $M$ and a connection on its determinant line bundle. (See Subsection \ref{sec rho shift} for questions about $\rho$-shifts in this context.)


\subsection{Reduction at nonzero values of $\munab$} \label{sec cocpt nonzero}

Landsman's conjecture was extended to reduction at $K$-theory classes corresponding to nontrivial representations in \cite{HochsDS, HochsPS}. Here one works with reduced quantisation, with values in $K_*(C^*_{\red}G)$. 

Because we will deduce the result in this subsection from Paradan and Vergne's result in \cite{PV}, we now adopt their convention concerning the definition of the momentum map:
\[
-\frac{i}{2}\munab_X  = \nabL_{X^M} - \cL^L_{X}.
\]
I.e., the factor $2\pi i$ in \eqref{eq def Spinc mom map}, which was chosen for consistency with \cite{HM, TZ98}, is replaced by $-i/2$. We use this convention in the present subsection, and in Section \ref{sec cocpt}.

Suppose $G$ is almost connected, and let $K<G$ be a maximal compact subgroup. With notation as in Subsection \ref{sec cpt}, one has
\[
R(K) = \bigoplus_{\lambda \in \Lambda_+ + \rho_K} \Z [\pi^K_{\lambda}]. 
\]
Set $d := \dim(G/K)$.
By the Connes--Kasparov conjecture, proved in \cite{CEN} for almost connected groups, the \emph{Dirac induction} map
\[
\DInd_K^G: R(K) \to K_d(C^*_{\red}G)
\]
is an isomorphism of Abelian groups, while $K_{d+1}(C^*_{\red}G) = 0$. In other words, the $K$-theory group $K_*(C^*_{\red}G)$ is the free Abelian group generated by 
\begin{equation} \label{eq class lambda}
[\lambda] := \DInd_K^G[\pi^K_{\lambda}],
\end{equation}
for $\lambda \in \Lambda_+ + \rho_K$, and these generators have degree $d$. For $G$ semisimple with discrete series, `most' of the generators $[\lambda]$ are associated to discrete series representations \cite{Lafforgue}. If $G$ is complex-semisimple, they are associated to families of principal series representations \cite{PP}. See also \cite{HochsPS}. 

Since $K_{d+1}(C^*_{\red}G) = 0$, it follows that $Q^{\Spinc}_G(M)_{\red} = 0$ if $d_M$ and $d$ have different parities. (Recall that we set $d_M := \dim(M)$.) So assume $d_M - d$ is even. In \cite{HochsPS}, the case where $M$ carries a (pre)symplectic form was considered. It was conjectured that quantisation commutes with reduction at any $\lambda \in \Lambda_+ + \rho_K$, in the sense that
\begin{equation} \label{eq conj [Q,R]=0 cocpt nonzero}
Q^{\Spinc}_G(M)_{\red} = \sum_{\lambda \in \Lambda_+ + \rho_K} Q(M_{\lambda}) [\lambda] \quad \in K_d(C^*_{\red}G).
\end{equation}
It was assumed that the momentum map image has nonzero intersection with the interior of a positive Weyl chamber, to simplify the $\rho$-shifts that occur (analogously to the way Theorem \ref{thm [Q,R]=0 cpt} simplifies to Corollary \ref{cor cpt Ab stab}). We will not make this assumption in Theorem \ref{thm [Q,R]=0 cocpt}.

In the symplectic setting, a formal version of quantisation, defined as the right hand side of \eqref{eq conj [Q,R]=0 cocpt nonzero}, was extended to non-cocompact actions and studied in \cite{HM2}.



Replacing $G$ by a double cover if necessary, we may assume the lift \eqref{eq tilde Ad} of the adjoint action by $K$ on $\kp$ exists. Let the slice $N \subset M$ and the $\Spinc$-structure $P_N \to N$ be as in Proposition \ref{prop slice spinc}. Since $M/G$ is compact, $N$ is compact in this case. We choose a connection $\nabla^M$ on $L^M$ such that \eqref{eq mu N M 1} holds.

To quantise singular reduced spaces, we extend Definition \ref{def QR} by using the homeomorphism of Lemma \ref{lem Nred Mred} and Paradan and Vergne's definition in the singular case. Recall that $\mu^{\nabla^N}$ is the $\Spinc$-momentum map for the action by $K$ on $N$.
\begin{definition} \label{def QR sing}
If $\xi \in \kk^*$ is a singular value of $\mu^{\nabla^N}$, then
\[
Q^{\Spinc}(M_{\xi}) := Q^{\Spinc}(N_{\xi}),
\] 
where $Q^{\Spinc}(N_{\xi})$ is defined as in Section 4 of \cite{PV}.
\end{definition}
Note that different choices of $N$ lead to homeomorphic reduced spaces by Lemma \ref{lem Nred Mred}.
If $\xi$ is a \emph{regular} value of $\mu^{\nabla^N}$, then Definition \ref{def QR} applies by Proposition \ref{prop Spinc fibred}. Because of Proposition \ref{prop PNxi PMxi}, one has $Q^{\Spinc}(M_{\xi}) = Q^{\Spinc}(N_{\xi})$ in that case, so Definitions \ref{def QR} and \ref{def QR sing} are consistent.

\begin{remark}
Another way of phrasing Definition \ref{def QR sing} is that for any (singular) value $\xi \in \kk^*$  of $\mu^{\nabla^N}$, the shifted element $\xi + \varepsilon \in \kk^*$ is a regular value for generic $\varepsilon$ in a linear subspace of $\kk^*$, and hence a $\Spinc$-regular value of $\munabM$ by Proposition \ref{prop Spinc fibred}. Furthermore, the quantisation $Q^{\Spinc}(M_{\xi+ \varepsilon})$ is independent of such $\varepsilon$ close enough to $0$. Hence one can define 
\[
Q^{\Spinc}(M_{\xi}) := Q^{\Spinc}(M_{\xi+ \varepsilon})
\]
for such $\varepsilon$.

To see that this is true, note that, as mentioned above, Proposition \ref{prop PNxi PMxi} implies that
\[
Q^{\Spinc}(M_{\xi+ \varepsilon}) = Q^{\Spinc}(N_{\xi+ \varepsilon})
\]
if $\xi + \varepsilon$ is a regular value of $\munabN$. The claim therefore follows from Theorem 5.4 in \cite{PV} if $K$ is a torus, and from the arguments in Section 5.3 in \cite{PV} for general $K$.
\end{remark}

Paradan and Vergne's result generalises to the cocompact setting in the following way.
\begin{theorem}[$\Spin^c$ quantisation commutes with reduction; cocompact case] \label{thm [Q,R]=0 cocpt}
If $M$ and $G$ are connected, and  $d_M - d$ is even, then
\begin{equation} \label{eq [Q,R]=0 cocpt}
Q^{\Spinc}_{G}(M)_{\red} = \sum_{\lambda \in \Lambda_+ + \rho_K} m_{\lambda} [\lambda],
\end{equation}
with $m_{\lambda}$ given by \eqref{eq [Q,R]=0 cpt}.
\end{theorem}
This result will be proved in Section \ref{sec cocpt}. We will use the constructions in Subsections \ref{sec fibred prod} and \ref{sec slice red} and a quantisation commutes with induction result to deduce it from Paradan and Vergne's result.
In the symplectic setting, an additional assumption was needed in \cite{HochsDS} to 
 apply a similar kind of reasoning. The authors view this as a sign that it is very natural to study the quantisation commutes with reduction problem in the $\Spinc$-setting.

\section{$\Spinc$-structures on reduced spaces and fibred products} \label{sec cocpt}

In this section, we prove the statements in Section \ref{sec Spinc red slice}. Together with a generalisation of the \emph{quantisation commutes with induction} results in \cite{HochsDS, HochsPS}, this allows us to deduce
Theorem \ref{thm [Q,R]=0 cocpt} from Paradan and Vergne's result, Theorem \ref{thm [Q,R]=0 cpt}. Note that in Subsections \ref{sec Spinc red} and \ref{sec fibred prod}, group actions were not asumed to be cocompact. So the statements made there apply more generally (and many will also be used in Part \ref{part non-cocpt}). The cocompactness assumption will only be made in Subsection \ref{sec [Q,Ind]=0}.

Proposition \ref{prop Spinc reg}, Lemma \ref{lem Nred Mred} and Propositions \ref{prop Spinc fibred} and \ref{prop PNxi PMxi} are proved in Subsections  \ref{sec red reg}--\ref{sec PNxi PMxi}.
In Subsection \ref{sec [Q,Ind]=0}, we show that quantisation commutes with induction in the $\Spinc$-setting, and use this to prove Theorem \ref{thm [Q,R]=0 cocpt}.


\subsection{$\Spinc$-reduction at regular values} \label{sec red reg}

We start by proving Proposition \ref{prop Spinc reg}. Suppose $\xi \in \kg^*$ is a regular value of $\munab$. 
Then by Lemma \ref{lem loc free}, $G_{\xi}$ acts locally freely on $\munabinv(\xi)$. 
Let $q: \munabinv(\xi) \to M_{\xi}$ be the quotient map. The restriction of $TM$ to $\munabinv(\xi)$ decomposes as follows.
\begin{lemma} \label{lem TM restr}
There is a $G_{\xi}$-equivariant isomorphism of vector bundles
\begin{equation} \label{eq TM restr}
TM|_{\munabinv(\xi)} = q^*TM_{\xi} \oplus \kg^* \oplus \kg_{\xi},
\end{equation}
where $G_{\xi}$ acts on the right hand side by
\[
g\bigl( (m, v), \eta, X\bigr) = \bigl( (gm, v), \Ad^*(g) \eta, \Ad(g)X\bigr), 
\]
for $g \in G_{\xi}$, $m \in \munabinv(\xi)$, $v \in T_{G_{\xi}\cdot m}M_{\xi}$, $\eta \in \kg^*$ and $X \in \kg_{\xi}$.
\end{lemma}
\begin{proof}
See (5.6) in \cite{GGK} for the case where $G$ is a torus. In general, since $\xi$ is a regular value of $\munab$, 
we have the short exact sequence
\begin{equation} \label{eq SES 1}
0 \to \ker (T\munab) \to TM|_{\munabinv(\xi)} \xrightarrow{T\munab} \munabinv(\xi) \times \kg^* \to 0.
\end{equation}
Now $\ker(T\munab) = T\bigl(\munabinv(\xi)\bigr)$ fits into the short exact sequence
\begin{equation} \label{eq SES 2}
0 \to \ker(Tq) \to T\bigl(\munabinv(\xi)\bigr) \xrightarrow{Tq} TM_{\xi} \to 0.
\end{equation}
Since $\ker(Tq)$ is the bundle of tangent spaces to $G_{\xi}$-orbits, and $\kg_{\xi}$ acts locally freely on $\munabinv(\xi)$ by Lemma \ref{lem loc free}, we have 
\begin{equation} \label{eq ker Tq}
\ker(Tq) \cong \munabinv(\xi) \times \kg_{\xi},
\end{equation}
via the map
\[
(m, X) \mapsto X^M_m,
\]
for $(m, X) \in \munabinv(\xi) \times \kg_{\xi}$.

Combining \eqref{eq SES 1}, \eqref{eq SES 2} and \eqref{eq ker Tq}, we obtain the desired vector bundle isomorphism.
\end{proof}

Because of Lemma \ref{lem TM restr}, Proposition \ref{prop Spinc reg} follows from the following fact.
\begin{lemma}
If the conditions in Proposition \ref{prop Spinc reg} hold, then
 the sub-bundle
\begin{equation} \label{eq gstar gxi bdle}
\munabinv(\xi) \times (\kg^* \oplus \kg_{\xi}) \to \munabinv(\xi)
\end{equation}
of \eqref{eq TM restr}
has a $G_{\xi}$-equivariant $\Spin$-structure. 
\end{lemma}
\begin{proof}
Using the given $\Ad(G_{\xi})$-invariant, nondegenerate bilinear form on $\kg$, and the subspace $V \subset \kg$, we obtain an $\Ad(G_{\xi})$-equivariant isomorphism
\begin{equation} \label{eq decomp gstar gxi}
\kg^* \oplus \kg_{\xi} \cong  (\kg_{\xi} \oplus \kg_{\xi}) \oplus   V.
\end{equation}
Identifying $\kg_{\xi} \oplus \kg_{\xi} \cong \kg_{\xi} + i\kg_{\xi} = (\kg_{\xi})_{\C}$, and using the given complex structure on $V$, one gets an $\Ad(G_{\xi})$-invariant complex structure on \eqref{eq decomp gstar gxi}. This induces a $G_{\xi}$-equivariant $\Spinc$-structure on the vector bundle \eqref{eq gstar gxi bdle}, with determinant line bundle
\begin{equation} \label{eq det gxi}
\munabinv(\xi) \times {\bigwedge}_{\C}^{\topp} \bigl((\kg_{\xi})_{\C} \oplus V\bigr) \to \munabinv(\xi).
\end{equation}
Since $G$ and $G_{\xi}$ are unimodular, the adjoint action by $G_{\xi}$ on $\kg$, $\kg_{\xi}$ and hence $V$, has determinant one. Therefore, $G_{\xi}$ acts trivially on 
\[
{\bigwedge}_{\C}^{\topp} (\kg_{\xi})_{\C} \otimes {\bigwedge}_{\C}^{\topp} V   = {\bigwedge}_{\C}^{\topp} \bigl((\kg_{\xi})_{\C} \oplus V\bigr), 
\] 
so that the determinant line  bundle \eqref{eq det gxi} is equivariantly trivial. Hence the $\Spinc$-structure on $\eqref{eq gstar gxi bdle}$ is induced by a $G$-equivariant $\Spin$-structure. (Compare this with the fact that the natural embedding of $\U(n)$ into $\Spinc(2n)$ maps $\SU(n)$ into $\Spin(2n)$.)
\end{proof}


 \subsection{Induced connections and momentum maps}  \label{sec ind nabla} 

In the rest of this section, we fix a slice $N \subset M$ and a $K$-equivariant $\Spinc$-structure $P_N \to N$ as in Proposition \ref{prop slice spinc}. 

To prove Lemma \ref{lem Nred Mred}, we will choose the connection $\nabla^M$ in such a way that the $\Spinc$-momentum maps are related as in \eqref{eq mu N M 1}.
Let $\nabla^N$ be a $K$-equivariant Hermitian connection on the determinant line bundle $L^N \to N$.  We will use the connection $\nabla^M$ on $L^M = G\times_K L^N$ induced by $\nabla^N$, as discussed in
Section 3.1 in \cite{HochsDS}. We briefly review the construction of this connection.

Let $p_N: G \times N \to N$ be projection onto the second factor.
For a $K$-invariant section $s \in \Gamma^{\infty}(G\times N, p_N^*L^N)^K$, one has the section $\sigma \in \Gamma^{\infty}(L^M)$ given by
\begin{equation} \label{eq s sigma}
\sigma[g, n] = [g, s(g, n)].
\end{equation}
(Here $s$ is viewed as a map $G\times N \to L^N$.) For such an $s$, and for $g \in G$ and $n \in N$, write
\[
s_g(n) := s(g, n) =: s^n(g) \quad \in L^N_n.
\]
This defines $s_g \in \Gamma^{\infty}(L^N)$ and $s_n \in C^{\infty}(G, L^N_n) \cong C^{\infty}(G)$.

Let $q: G\times N \to M$ be the quotient map. Note that
\[
q^*L^M \cong p_N^*L^N \cong G\times L^N \to G\times N,
\]
and that under this isomorphism, $q^*\sigma$ corresponds to $s$. For $X \in \kg$, $n \in N$ and $v \in T_nN$, one has
\[
Tq(X, v) \in T_{[g, n]}M.
\]
Write $X = X_{\kk} + X_{\kp}$ according to the decomposition $\kg = \kk \oplus \kp$. Then the connection $\nabla^M$ is defined by the properties that it is $G$-invariant, and satisfies
\begin{equation} \label{eq def nabla M}
\bigl(\nabla^M_{Tq(X, v)}\sigma\bigr)[e, n] = \bigl[ e,  (\nabla^N_v s_e)(n) + X(s^n)(e) + 2\pi i \mu^{\nabla^N}_{X_{\kk}}(n)s(e,n) \bigr],
\end{equation}
for $X \in \kg$, $n \in N$, $v \in T_nN$, and $\sigma$ and $s$ as above.

Let $\munabN: N \to \kk^*$ be the $\Spinc$-momentum map associated to $\nabla^N$, and let 
$\munabM: M \to \kg^*$ be the $\Spinc$-momentum map for the induced connection $\nabla^M$.
Lemma \ref{lem Nred Mred} follows directly from the relation \eqref{eq mu N M 1} between $\munabN$ and $\munabM$, which holds because of the fourth point in Lemma \ref{lem L N M} and the following fact.
\begin{lemma} \label{lem mu N M}
For all $n \in N$, one has $\munabM(n) \in \kk^*$. 
\end{lemma}
Recall that we consider $\kk^*$ as a subspace of $\kg^*$ by identifying it with the annihilator of $\kp$.
\begin{proof}
As in \eqref{eq s sigma}, let $s \in \Gamma^{\infty}(G\times N, p_N^*L^N)^K$, and let $\sigma \in \Gamma^{\infty}(L^M)$ be the associated section of $L^M$. Let $X \in \kg$, and $n \in N$. Then one has
%
\[
\begin{split}
(\cL^{L^M}_X \sigma)[e,n] &= \ddt \exp(tX)[\exp(-tX), s(\exp(-tX), n)] \\
	& = \ddt [e, s(\exp(-tX), n)] \\
	&= [e, X(s^n)(e)] \quad \in L^M_{[e,n]}.
\end{split}
\]
Since $Tq(X, 0) = X^M$ in \eqref{eq def nabla M}, one therefore has
\[
(\nabla^M_{X^M}\sigma)[e,n] = (\cL^{L^M}_X \sigma)[e,n] + 2\pi i \mu^{\nabla^N}_{X_{\kk}}(n) \sigma[e,n].
\]
Here $X = X_{\kk} + X_{\kp}$ according to the  decomposition $\kg = \kk \oplus \kp$. The claim follows.
\end{proof}

%
%
%

\subsection{$\Spinc$-reduction for fibred products} \label{sec fibred} 

We now turn to a proof of Proposition \ref{prop Spinc fibred}.
For any group $H$ acting on a manifold $Y$, 
we use the notation $q_H$ for the quotient map $Y \to Y/H$. If $H < K$, we will write $\kp_Y$ for the trivial bundle $Y \times \kp \to Y$, on which $H$ acts via the adjoint representation on $\kp$.

\medskip
\noindent
\emph{Proof of Proposition \ref{prop Spinc fibred}.}
Let $\xi \in \kk^*$ be a regular value of $\munabN$. 
Since \eqref{eq mu N M 1} holds, we have
\begin{equation} \label{eq decomp mu Gxi}
\munabMinv(G\xi) = G\times_K \munabNinv(K\xi).
\end{equation}
Because of this relation, it will be convenient to initially consider the restriction of $TM$ to $\munabMinv(G\xi)$, rather than to $\munabMinv(\xi)$. 
Let
\begin{equation} \label{eq TN restr}
TN|_{\munabNinv(K\xi) }= q_K^* TN_{\xi} \oplus \cN^{K\xi}_{N}
\end{equation}
be a $K$-invariant splitting. 
By Lemma \ref{lem normal N M} below, 
we have a $G$-invariant splitting
\[
TM|_{\munabinv(G\xi)} = q_G^*TM_{\xi} \oplus \cN^{G\xi}_M,
\]
with 
\[
\cN^{G\xi}_M =  \bigl(G\times_K \cN^{K\xi}_{N} \bigr) \oplus  \bigl(  G\times_K\kp_{\munabNinv(K\xi)} \bigr).
\]

By Lemma \ref{lem TN restr}, the vector bundles
\[
G\times_K \cN^{K\xi}_N
\]
and 
\[
G\times_K \kp_{\munabNinv(K\xi)} 
\]
over $\munabMinv(G\xi) = G\times_K \munabNinv(K\xi)$
have  $G$-equivariant $\Spin$-structures. 
By Lemma \ref{lem 2 out of 3} and Remark \ref{rem G 2 out of 3}, these induce a $G$-equivariant $\Spinc$-structure on $\cN^{G\xi}_M$ with equivariantly trivial determinant line bundle, i.e.\ a $G$-equivariant $\Spin$-structure. 
%
Restricting all bundles from $\munabMinv(G\xi)$ to $\munabMinv(\xi)$, and group actions from $G$ to $G_{\xi}$, we  obtain a $G_{\xi}$-equivariant splitting
\[
TM|_{\munabinv(\xi)} = q_{G_{\xi}}^*TM_{\xi} \oplus \cN^{\xi}_M,
\]
where $\cN^{\xi}_M$ has a $G_{\xi}$-equivariant $\Spin$-structure. 
\hfill $\square$


It remains to prove Lemmas \ref{lem normal N M} and \ref{lem TN restr}, used in the proof of Proposition \ref{prop Spinc fibred}.

\begin{lemma} \label{lem normal N M}
One has
\begin{equation} \label{eq restr TM lem}
TM|_{\munabinv(G\xi)} = q_G^*TM_{\xi} \oplus \cN^{G\xi}_M,
\end{equation}
with 
\[
\cN^{G\xi}_M =  \bigl(G\times_K \cN^{K\xi}_{N} \bigr) \oplus  \bigl(  G\times_K\kp_{\munabNinv(K\xi)} \bigr),
\]
and $\cN^{K\xi}_{N}$ as in \eqref{eq TN restr}.
\end{lemma}
\begin{proof}
%
Because of \eqref{eq decomp TM 1} and \eqref{eq decomp mu Gxi}, we see that
\[
\begin{split}
TM|_{\munabMinv(G\xi)} &= G\times_K \bigl( TN|_{\munabNinv(K\xi)} \oplus \kp_{\munabNinv(K\xi)} \bigr) \\
	&= G\times_K \bigl( q_K^*TN_{\xi} \oplus \cN^{K\xi}_N \oplus \kp_{\munabNinv(K\xi)} \bigr) \\
	&= q_G^*TM_{\xi} \oplus \bigl(G\times_K \cN^{K\xi}_N\bigr) \oplus \bigl( G\times_K \kp_{\munabNinv(K\xi)} \bigr).
\end{split}
\]
\end{proof}

\begin{lemma} \label{lem TN restr} 
For a choice of the bundle $\cN^{K\xi}_N$ as in \eqref{eq TN restr},  and hence for any such bundle,  the vector bundles
\[
G\times_K \cN^{K\xi}_N 
\]
and
\[
G\times_K \kp_{\munabNinv(K\xi)} 
\]
over $\munabMinv(G\xi) = G\times_K \munabNinv(K\xi)$ have  $G$-equivariant $\Spin$-structures. 
\end{lemma}
\begin{proof}
Since $K$ is compact, and $\xi$ is a regular value of $\munabN$, Proposition \ref{prop Spinc reg}
and Example \ref{ex reg se} imply that
\[
TN|_{\munabNinv(\xi)} = q_{K_{\xi}}^* TN_{\xi} \oplus \cN^{\xi}_{N}, 
\]
where $\cN^{\xi}_{N}$ has a  $K_{\xi}$-equivariant $\Spin$-structure $P^{\xi}_N$. 
Set
\[
\cN^{K\xi}_N := K\cdot \cN^{\xi}.
\]
Then we have a  $K$-equivariant vector bundle isomorphism
\[
K \times_{K_{\xi}} \cN^{\xi}_N \cong \cN^{K\xi}_N,
\]
given by $[k, v] \mapsto T_n k(v)$, for $n \in \munabNinv(\xi)$, $v \in (\cN^{\xi}_{N})_n$ and $k \in K$. This extends to a $G$-equivariant isomorphism
\begin{equation} \label{eq fibred normal}
G \times_{K_{\xi}} \cN^{\xi}_N \cong G\times_K \cN^{K\xi}_N
\end{equation}
Now
\[
P_N^{G\xi} := G\times_{K_{\xi}} P_N^{\xi} \to G\times_{K_{\xi}} \munabNinv(\xi) \cong \munabMinv(G\xi)
\]
defines a $\Spin$-structure on \eqref{eq fibred normal}. 

Furthermore, since the adjoint action by $K$ on $\kp$ lifts to $\Spin(\kp)$, 
the vector bundle $\kp_{\munabNinv(K\xi)}$ has a $K$-equivariant $\Spin$-structure 
\[
\munabNinv(K\xi) \times \Spin(\kp).
\]
As above, this induces a $G$-equivariant $\Spin$-structure on 
\[
G\times_K \kp_{\munabNinv(K\xi)} \to \munabMinv(G\xi).
\]
\end{proof}

\subsection{$\Spinc$-structures on $N_{\xi}$ and $M_{\xi}$} \label{sec PNxi PMxi}

The last statement from Section \ref{sec Spinc red slice} we prove is Proposition \ref{prop PNxi PMxi}. 
As before, let $\xi \in \kk^*$ be a regular value of $\munabN$, and let 
let the $\Spinc$-structure $P_N \to N$ be as in Proposition \ref{prop slice spinc}.
To prove Proposition \ref{prop PNxi PMxi}, we must show that the $\Spinc$-structures induced on $N_{\xi}$ and $M_{\xi}$, induced by $P_N$ and $P_M$ respectively, via Propositions \ref{prop Spinc reg} and \ref{prop Spinc fibred}, coincide. 

We first give a slightly different description of $\Spinc$-structures induced on reduced spaces from the expression \eqref{eq PMxi}.
\begin{lemma} \label{lem spinc red orbit}
In the setting
 of Lemma \ref{lem Spinc normal}, the $\Spinc$-structure $P_{M_{\xi}}$ induced on $M_{\xi}$ equals
\[
P_{M_{\xi}}  = \Destab_{\cN^{G\xi}}\bigl(P_M|_{\munabinv(G\xi)} \bigr)/G,
\]
where $\cN^{G\xi} \to \munabinv(G\xi)$ is a vector bundle with the property of $\cN^{G\xi}_M$ in \eqref{eq restr TM lem}, and with a $G$-equivariant $\Spin$-structure. 
\end{lemma}
\begin{proof}
By \eqref{eq PMxi} and Lemma \ref{lem stab fibred}, we have
\[
\begin{split}
P_{M_{\xi}}  &= \Destab_{\cN^{\xi}}\bigl(P_M|_{\munabinv(\xi)} \bigr)/G_{\xi} \\
	&= \bigl(G\times_{G_{\xi}}   \Destab_{\cN^{\xi}}\bigl(P_M|_{\munabinv(\xi)} \bigr)   \bigr)/G \\
&= \Destab_{G\times_{G_{\xi}}\cN^{\xi}}\bigl( G\times_{G_{\xi}}\bigl(P_M|_{\munabinv(\xi)}\bigr) \bigr)/G.
\end{split}
\]
Here $\cN^{\xi}\to \munabinv(\xi)$ has a $G_{\xi}$-equivariant $\Spin$-structure $P_{\cN^{\xi}}$. 

Similarly to the proof of Lemma \ref{lem TN restr}, set $\cN^{G\xi} := G\cdot \cN^{\xi}$. Then
\[
G\times_{G_{\xi}}\cN^{\xi} \cong \cN^{G\xi}.
\]
The left hand side has the $G$-equivariant $\Spin$-structure $G\times_{G_{\xi}}P_{\cN^{\xi}}$. 
Since also
\[
G\times_{G_{\xi}}\bigl(P_M|_{\munabinv(\xi)}\bigr) \cong P_M|_{\munabinv(G\xi)}.
\]
the claim follows.
\end{proof}

\noindent \emph{Proof of Proposition \ref{prop PNxi PMxi}.}
Let $P_{N_{\xi}}\to N_{\xi}$ be the $\Spinc$-structure on $N_{\xi}$ induced by $P_N$ because of Proposition \ref{prop Spinc reg}, and let $P_{M_{\xi}}\to M_{\xi}$ be the $\Spinc$-structure on $M_{\xi}$ induced by $P_M$ because of Proposition \ref{prop Spinc fibred}.
We saw in Proposition \ref{prop slice spinc} that
\[
P_M = G\times_K \Stab_{\kp_N}(P_N).
\]
Let $\cN^{G\xi}_M$ and $\cN^{K\xi}_N$ be as in Lemma \ref{lem normal N M}. Then, by Lemma \ref{lem spinc red orbit},
\[
\begin{split}
P_{M_{\xi}} &=  \Destab_{\cN^{G\xi}_M}\bigl(P_M|_{\munabMinv(G\xi)} \bigr)/G \\
	&=  \Destab_{\cN^{G\xi}_M}\bigl(   \bigl(G\times_K \Stab_{\kp_N}(P_N) \bigr)    |_{\munabMinv(G\xi)} \bigr)/G \\
	&=  \Destab_{\cN^{G\xi}_M}\bigl(  \Stab_{G\times_K \kp_N}
	 \bigl(G\times_K (P_N|_{\munabNinv(K\xi)}) \bigr)     \bigr)/G \\
	&=  \Destab_{G\times_K \cN^{K\xi}_N}\bigl( 
	 G\times_K (P_N|_{\munabNinv(K\xi)})     \bigr)/G.
\end{split}
\]
In the third equality, we have used  the first point of Lemma \ref{lem stab fibred} and \eqref{eq decomp mu Gxi}. In the last equality, we applied  Lemmas \ref{lem destab stab} and \ref{lem normal N M}. By the second point of Lemma \ref{lem stab fibred}, we conclude that
\[
\begin{split}
P_{M_{\xi}} &=    G\times_K \bigl( \Destab_{ \cN^{K\xi}_N} (P_N|_{\munabNinv(K\xi)})    \bigr)  /G \\
 &=  \bigl( \Destab_{ \cN^{K\xi}_N} (P_N|_{\munabNinv(K\xi)})    \bigr)  /K \\
 &= P_{N_{\xi}},
\end{split}
\]
by Lemma \ref{lem spinc red orbit} (applied to the action by $K$ on $N$).
\hfill $\square$

\subsection{Quantisation commutes with induction}\label{sec [Q,Ind]=0}

Together with the constructions of $\Spinc$-structures proved so far in this section, the quantisation commutes with induction techniques of \cite{HochsDS, HochsPS} allow us to deduce Theorem \ref{thm [Q,R]=0 cocpt} from Paradan and Vergne's result, Theorem \ref{thm [Q,R]=0 cpt}. 

We now suppose that $M/G$, and hence $N$ is \emph{compact}.
The connections $\nabla^N$ and $\nabla^M$ induce
 Dirac operators on $N$ and $M$, which can be used to define the quantisations of these manifolds. After the  quantisation commutes with induction results of \cite{HochsDS} (in the symplectic setting) and \cite{HochsPS} (in the presymplectic setting), the following $\Spinc$-version of this principle is perhaps the most natural and general.
\begin{theorem}[$\Spinc$-quantisation commutes with induction] \label{thm [Q,Ind]=0}
In the setting of Proposition \ref{prop slice spinc},
the Dirac induction map $\DInd_K^G$ maps the $\Spinc$-quantisation of $N$ to the $\Spinc$-quantisation of $M$:
\[
\DInd_K^G\bigl( Q^{\Spinc}_{K}(N)\bigr) = Q^{\Spinc}_{G}(M)_{\red} \quad \in K_*(C^*_{\red}G).
\]
\end{theorem}
\begin{proof}
Let $K_*^K(N)$ and $K_*^G(M)$ be the equivariant $K$-homology groups \cite{BCH} of $N$ and $M$, respectively.
In Theorem 4.6 in \cite{HochsDS} and Theorem 4.5 in \cite{HochsPS}, a map
\[
\KInd_K^G: K_*^K(N) \to K_*^G(M)
\]
is constructed, such that the following diagram commutes:
\[
\xymatrix{
K_*^G(M) \ar[rr]^-{r_* \circ \Gind} && K_*(C^*_{\red}G) \\
K_*^K(N) \ar[u]^-{\KInd_K^G} \ar[rr]^-{\Kind} && R(K) \ar[u]_-{\DInd_K^G}.
}
\]
Here, as before, $\Gind$ is the analytic assembly map. The map $\Kind$ is the analytic assembly map for the action by $K$ on $N$, which coincides with the usual equivariant index.

In Section 6 of \cite{HochsDS}, it is shown that the map $\KInd_K^G$ maps the class
\[
[D_N] \in K_0^K(N)
\]
of to the $\Spinc$-Dirac operator $D_N$ on $N$, to the class
\[
[D_M] \in K_d^G(M)
\]
of the $\Spinc$-Dirac operator $D_M$ on $M$. Although in \cite{HochsDS} the symplectic setting is considered, the arguments in Section 6 of that paper are stated purely in terms of $\Spinc$-structures. Hence they apply in this more general setting, and we conclude that
\[
\begin{split}
\DInd_K^G\bigl( Q^{\Spinc}_{K}(N)\bigr)  &= \DInd_K^G \bigl( \Kind[D_N] \bigr) \\
	&= r_* \circ \Gind \bigl(  \KInd_K^G[D_N] \bigr) \\
	&= r_* \circ \Gind[D_M] \\
	&= Q^{\Spinc}_{G}(M)_{\red}.
\end{split}
\]
\end{proof}

Theorem \ref{thm [Q,R]=0 cocpt} follows by combining Theorem \ref{thm [Q,Ind]=0}, Proposition \ref{prop slice spinc}, Proposition \ref{prop PNxi PMxi}, and Paradan and Vergne's Theorem \ref{thm [Q,R]=0 cpt}.

\medskip
\noindent \emph{Proof of Theorem \ref{thm [Q,R]=0 cocpt}.}
By Proposition \ref{prop slice spinc}, Theorem \ref{thm [Q,Ind]=0} and
Theorem \ref{thm [Q,R]=0 cpt}, we have
\[
Q^{\Spinc}_{G}(M)_{\red} = \DInd_K^G\bigl( Q^{\Spinc}_{K}(N)\bigr)  = \sum_{\lambda \in \Lambda_+ + \rho_K} m_{\lambda}[\lambda],
\]
with $m_{\lambda}$ as in \eqref{eq [Q,R]=0 cpt}, where $Q^{\Spinc}(M_{\xi})$ is replaced by $Q^{\Spinc}(N_{\xi})$ for all $\xi$ that occur. By Definition \ref{def QR sing}, these two quantisations are equal if $\xi$ is a singular value of $\mu^{\nabla^N}$. If $\xi$ is a regular value of this map, they are equal by Proposition \ref{prop PNxi PMxi}, and the claim follows.
\hfill $\square$
\medskip



\part{Non-cocompact actions} \label{part non-cocpt}


\section{The results on non-cocompact actions} \label{sec non-cocpt case}

The main result in this paper for untwisted $\Spinc$-Dirac operators, for possibly non-cocompact actions and reduction at zero, is Theorem \ref{thm [Q,R]=0 large p}. We state it in Subsection \ref{sec non-cocpt}, and prove it in Sections \ref{sec Bochner} and \ref{sec loc est}. The generalisation of this result to $\Spinc$-Dirac operators twisted by vector bundles, Theorem \ref{thm index}, is stated in Subsection \ref{sec result twist}. It is proved in Section \ref{sec twisted}.

While the proof of Theorem \ref{thm [Q,R]=0 cocpt} in Section \ref{sec cocpt} was based on Paradan and Vergne's result in \cite{PV}, our proofs of Theorems \ref{thm [Q,R]=0 large p} and \ref{thm index} are independent of their result.

To state a $\Spinc$-quantisation commutes with reduction result without assuming that $M/G$ is compact, we recall some facts about the $G$-invariant, transversally $L^2$-index introduced in Section 4 of \cite{HM}. We now suppose that $G$ is \emph{unimodular}, and fix a left- and right-invariant Haar measure $dg$ on $G$.

\subsection{The invariant, transversally $L^2$-index}

 The definition of the invariant, transversally $L^2$-index involves \emph{cutoff functions}.
\begin{definition} \label{def cutoff fn}
Let $G$ be a unimodular locally compact group acting properly on a locally compact Hausdorff space $X$. A \emph{cutoff function} is a continuous function $f$ on $X$ such that the support of $f$ intersects every $G$-orbit in a compact set, and for all $x \in X$, one has
\[
\int_G f(gx)^2\, dg = 1,
\]
with respect to a Haar measure $dg$ on $G$.
\end{definition}
It is shown in Proposition 8 in Section 2.4 of Chapter 7 in \cite{Bourbaki} that cutoff functions exist.

Let $E \to M$ be a $G$-equivariant vector  bundle, equipped with a $G$-invariant metric. Let 
$L^2(E)$ be the $L^2$-space of sections of $E$, with respect to this metric, and 
 the density on $M$ associated to the Riemannian metric induced by the $\Spinc$-structure.
\begin{definition}
The space $L^2_T(E)$ of \emph{transversally $L^2$-sections} of $E$ is the space of measurable sections $s$ of $E$ such that $fs \in L^2(E)$ for all cutoff functions $f$ on $M$, up to equality almost everywhere.
\end{definition}
One can show that for a $G$-invariant transversally $L^2$-section $s \in L^2_T(E)^G$, the $L^2$-norm of $fs$ does not depend on the cutoff function $f$ (see Lemma 4.4 in \cite{HM}). This turns the $G$-invariant part $L^2_T(E)^G$ of $L^2_T(E)$ into a Hilbert space.

Let $D$ be a $G$-equivariant (differential) operator on $\Gamma^{\infty}(E)$. Suppose $E$ is $\Z_2$-graded, and that $D$ is odd with respect to this grading. 
\begin{definition} \label{def transv L2 index}
The \emph{transversally $L^2$-kernel} of $D$ is
\[
\ker_{L^2_T}(D) := \ker(D) \cap L^2_T(E).
\]
If the $G$-invariant part $\ker_{L^2_T}(D)^G$ of $\ker_{L^2_T}(D)$ is finite-dimensional, then the \emph{$G$-invariant, transversally $L^2$-index} of $D$ is the integer
\[
\ind^G_{L^2_T}(D) := \dim\bigl(\ker_{L^2_T}(D^+)^G \bigr) - \dim\bigl(\ker_{L^2_T}(D^-)^G \bigr),
\]
where $D^{\pm}$ is the restriction of $D$ to the even or odd part of $\Gamma^{\infty}(E)$.
\end{definition}

\begin{remark} \label{rem L2T cpt}
If $G$ is compact, then the transversally $L^2$-index of $D$ is the $G$-invariant part of its $L^2$-index. If $M/G$ is compact, then the transversally $L^2$-index of $D$ is the index of $D$ restricted to $G$-invariant smooth sections.
\end{remark}

\subsection{Invariant quantisation} \label{sec non-cocpt}

As shown in \cite{HM}, the transversally $L^2$-index of Definition \ref{def transv L2 index} allows one to make sense of quantisation and reduction without assuming $M$, $G$ or $M/G$ to be compact. There will only be a cocompactness assumption on the set of zeros of a vector field on $M$. This vector field is defined in terms of the momentum map and a \emph{family} of inner products on $\kg^*$, by which we mean a metric on the vector bundle
\[
\gstarM := M \times \kg^* \to M,
\]
 with a certain $G$-invariance property. Using such a family of inner products, rather than a single one, allows us to define a suitable $G$-invariant vector field, despite the fact that $\kg$ does not admit an $\Ad(G)$-invariant inner product in general.

Let $\{(\relbar, \relbar)_m\}_{m \in M}$ be a $G$-invariant metric on the vector bundle $\gstarM$, with respect to the $G$-action given by
\[
g\cdot(m, \xi) = (g\cdot m, \Ad^*(g)\xi),
\]
for $g \in G$, $m \in M$ and $\xi \in \kg^*$. Such a metric exists by Lemma 2.1 in \cite{HM}. Consider the map
\[
(\munab)^*: M \to \kg
\]
defined by
\begin{equation} \label{eq def mu star}
\langle \xi,  (\munab)^*(m)\rangle = \bigl( \xi,  \munab(m) \bigr)_m,
\end{equation}
for all $\xi \in \kg^*$ and $m \in M$. This induces a $G$-invariant vector field ${\vnab}$ on $M$, given by
\begin{equation} \label{eq def v}
\vnab_m := 2\bigl((\munab)^*(m) \bigr)^M_m = 2\ddt \exp\bigl(t (\munab)^*(m) \bigr)m,
\end{equation}
for $m \in M$. (The factor $2$ was included for consistency with \cite{HM, TZ98}.) A central assumption we make is that the critical set $\Crit({\vnab})$ of zeros of ${\vnab}$ is \emph{cocompact}. This implies that $M_0$ is compact.

Recall the definition of the Dirac operator $D_p$ in Subsection \ref{sec Dirac ops}, for a $p \in \N$. We will apply the invariant, transversally $L^2$-index to a Witten-type deformation of $D_p$.
\begin{definition} \label{def deformed Dirac}
For $p \in \N$ and $t \in \R$,  the \emph{deformed Dirac operator} $D_{p, t}$ is the operator
\[
D_{p, t} := D_p + \frac{it}{2}c({\vnab})
\]
on $\Gamma^{\infty}(\cS_p)$.
\end{definition}
Note that 
\[
D_{1, 1} = D + \frac{i}{2}c({\vnab}).
\]
In general, $D_{p, t}$ is $G$-equivariant, by $G$-invariance of $\vnab$. Suppose that $M$ is even-dimensional. Then $\cS_p$ is $\Z_2$-graded, and $D_{p, t}$ is odd with respect to this grading.

Suppose $M$ is complete in the Riemannian metric induced by the $\Spinc$-structure.
It turns out that in this non-cocompact setting, the invariant, transversally $L^2$-index of $D_{p, t}$ is well-defined for large enough $t$.
\begin{theorem} \label{thm quant well defd}
One can choose the metric on $\gstarM$ in such a way that 
for all $t \geq 1$, the $G$-invariant part of $\ker_{L^2_T}(D_{p, t})$ is finite-dimensional, for all $p \in \N$. 
\end{theorem}
This allows us to define the $G$-invariant part of $\Spinc$-quantisation.
\begin{definition} \label{def invar quant}
The \emph{$G$-invariant $\Spinc$-quantisation of $M$} with respect to the given $\Spinc$-structure, and the connection $\nabla$ on $L$, is
\[
Q^{\Spin^c}(M)^G := \ind^G_{L^2_T}(D_{1, 1}).
\]
\end{definition}

Suppose $0$ is a $\Spinc$-regular value of $\munab$. By Proposition \ref{prop Spinc reg} and Example \ref{ex reg zero}, this is true for example if $0$  is a regular value of $\munab$ and $G$ is semisimple or Abelian. Alternatively,  by Proposition \ref{prop Spinc fibred}, it is enough that $0$ is a regular value of a $\Spinc$-moment map $\munabN: N \to \kk^*$ on a $\Spinc$-slice $N$. 
Since $M_0$ is compact by cocompactness of $\Crit({\vnab})$, Definition \ref{def QR} applies, and one has
\[
Q^{\Spinc}(M_0) = \ind(D_{M_0}).
\]
Analogously to the symplectic case \cite{HM} and the compact case  \eqref{eq [Q,R]=0 cpt zero}, one expects $\Spinc$-quantisation to commute with reduction in this non-cocompact setting. We will prove the following version of this statement.
\begin{theorem}[$\Spinc$-quantisation commutes with reduction; non-cocompact case] \label{thm [Q,R]=0 large p}
Suppose $G$ acts freely\footnote{It will turn out that, for a natural choice of $\nabla'$ on the determinant line bundle of the $\Spinc$-structure used, the $\Spinc$-momentum maps for $\nabla$ and $\nabla'$ differ by a nonzero factor, so that the condition that $G$ acts freely on $\munabinv(0)$ is the same for the two connections.} on $\munabinv(0)$ (rather than just locally freely). Then there exists a $G$-equivariant $\Spin^c$-structure on $M$ and a connection on the corresponding determinant line bundle, such that, for these choices,
\begin{equation} \label{eq [Q,R]=0 non-cocpt}
Q^{\Spin^c}(M)^G = Q^{\Spinc}(M_0) \quad \in \Z.
\end{equation}
\end{theorem}
\begin{remark}
The choice of $\Spinc$-structure in Theorem \ref{thm [Q,R]=0 large p} amounts to taking large enough tensor powers of the determinant line bundle of a given $\Spinc$-structure. I.e.\ one starts with an initial $\Spinc$-structure $P\to M$ with determinant line bundle $L\to M$, and the result holds for $\Spinc$-structures with determinant line bundle $L^p\to M$, for $p$ large enough. 
So if $L$ is not a torsion class in $H^2(M; \Z)$, then the result holds for infinitely many $\Spinc$-structures. 

The connection on the determinant line bundle $L^p$ used can be any connection induced by a connection on $L$ (and the minimal value of $p$ depends on this inital connection on $L$).
\end{remark}

\begin{remark}
We could prove Theorem \ref{thm quant well defd} by referring to \cite{Braverman2} and using the elliptic regularity arguments in \cite{HM}. We will give an independent proof of finite-dimentionality of $\ker_{L^2_T}(D_{p, t})^G$, however, as a by-product of the localisation arguments needed to prove Theorem \ref{thm [Q,R]=0 large p}. 
\end{remark}

\subsection{$\rho$-shifts and asymptotic results} \label{sec rho shift}

If $M$ and $G$ are compact, one may take $t_0 = 0$ in Definition \ref{def invar quant}. Then $Q^{\Spin^c}(M)^G$  is the invariant part of  \eqref{eq def quant cpt}, which by \eqref{eq [Q,R]=0 cpt zero} equals $Q(M_{\rho_K})$. On the other hand,
Theorem \ref{thm [Q,R]=0 large p} states that, for a certain $G$-equivariant $\Spinc$-structure on $M$ and a connection on its determinant line bundle,
\[
Q^{\Spin^c}(M)^G = Q^{\Spinc}(M_0).
\]
Hence, apparently, one has
\begin{equation} \label{eq M0 Mrho}
 Q(M_0) = Q(M_{\rho_K})
 \end{equation}
 for this choice of $\Spinc$-structure and connection. 

This potential contradiction can be resolved, by noting that, for the $\Spinc$-structure and the connection $\nabla'$ used, one has
\[
\mu^{\nabla'} = p \mu^{\nabla},
\]
for a connection $\nabla$  on the determinant line bundle of a $\Spinc$-structure initially given, and a large enough integer $p$. (See \eqref{eq mu p} in the proof of Proposition \ref{prop Spinc p}.) For any $\xi \in \kg^*$, let $M_{\xi}$ and $M'_{\xi}$ be the reduced spaces at $\xi$ for the momentum maps $\mu^{\nabla}$ and $\mu^{\nabla'}$, respectively. Then
\[
M'_{\xi} = M_{\xi/p}.
\]
In particular, $M'_0 = M_0$, and $M'_{\rho_K} = M_{\rho_K/p}$.

The statement \eqref{eq M0 Mrho} is therefore that
\[
Q(M_{\rho_K/p}) = Q(M_0),
\]
for $p$ large enough. In the symplectic setting, this follows from the fact that $Q(M_{\xi})$ is independent of small variations of $\xi$ (see Theorem 2.5 in \cite{MS} if the action is free on $(\mu^{\nabla})^{-1}(\xi)$, or \cite{Zhang} for a holomorphic version). More generally, if $M$ is of the form $M = G\times_K N$ as in Subsection \ref{sec cocpt nonzero}, then by Proposition \ref{prop PNxi PMxi}, one has
\[
Q(M_{\xi}) = Q(N_{\xi}),
\]
which is independent of small variations of $\xi$ if $N$ is a compact Hamiltonian $K$-manifold (but $M$ is not necessarily symplectic). 

In the general non-cocompact setting of Subsection \ref{sec non-cocpt}, this leads one to expect that, if $\munab$ is $G$-proper (in the sense that the preimage of any cocompact set is cocompact), there is an open neighbourhood $U$ of $0$ in $\kg^*$, such that for all $\Spinc$-regular values $\xi \in U$ of $\munab$,
\[
Q(M_{\xi}) = Q(M_{0}).
\]

The above arguments show that, for `asymptotic' quantisation commutes with reduction results,  reduction at zero (or possibly a nearby regular value of the momentum map) is really the only natural case to consider.

\subsection{An index formula for twisted $\Spinc$-Dirac operators} \label{sec result twist}

The main results on $\Spinc$-Dirac operators in the non-cocompact case, Theorems \ref{thm quant well defd} and \ref{thm [Q,R]=0 large p}, generalise to $\Spinc$-Dirac operators twisted by arbitrary vector bundles. We use this to obtain an index formula for Braverman's analytic index of such operators, Theorem \ref{thm index}, expressing it in terms of characteristic classes on $M_0$. A potentially interesting feature of this formula is that it involves localisation to $(\munab)^{-1}(0)$. In the setting we consider, where the manifold $M$, the group $G$ acting on it, and the quotient $M/G$ may all be noncompact, it is unlikely that there is a topological expression for the index of (twisted) $\Spinc$-Dirac operators in terms of characteric classes on $M$. However, localisation to $\munabinv(0)$ allows us to still define a meaningful topological index, as an integral over the compact space $M_0$.

In the compact setting, the index of any elliptic operator on a $\Spinc$-manifold equals the index of a twisted $\Spinc$-Dirac operator. Hence index formulas for the latter kind of operators immediately generalise to the former. In the noncompact setting we consider here, such a principle is not (yet) available. Still, the index formula we obtain for twisted $\Spinc$-Dirac operators strongly suggests a more general underlying equality of topological and analytic indices.

Fix $p \in \N$. 
We retain all other notation used previously. In particular, we have the connection $\nabla^{\cS_p}$ on $\cS_p$, and the $\Spinc$-moment map $\munab: M \to \kg^*$ induced by a connection $\nabla$ on the determinant line bundle $L \to M$. In addition, consider a Hermitian, $G$-equivariant vector bundle $E \to M$. Let $\nabla^E$ be  a Hermitian, $G$-invariant connection on $E$. Consider the connection
\[
\nabla^{\cS_p \otimes E} := \nabla^{\cS_p} \otimes 1_E + 1_{\cS_p} \otimes \nabla^E
\]
on $\cS_p \otimes E$. 
\begin{definition} \label{def twisted Dirac}
The \emph{twisted $\Spinc$-Dirac operator} associated to $\nabla$ and $\nabla^E$ is the composition
\[
D^E_p: \Gamma^{\infty}(\cS_p \otimes E) \xrightarrow{\nabla^{\cS_p \otimes E}} \Omega^1(M; \cS_p \otimes E)
	\xrightarrow{c\otimes 1_E} \Gamma^{\infty}(\cS_p \otimes E).
\]
For $t \in \R$,  the \emph{deformed $\Spinc$-Dirac operator} twisted by $E$ via $\nabla^E$ is the operator
\[
D^E_{p, t} := D^E_{p} + \frac{it}{2}c(v^{\nabla}) \otimes 1_E,
\]
\end{definition}

Theorems \ref{thm quant well defd} and \ref{thm [Q,R]=0 large p} generalise to the operator $D^E_p$ as follows.
\begin{theorem} \label{thm index}
Suppose that $0$ is a $\Spinc$-regular value of $\munab$, and that $G$ acts freely on $(\munab)^{-1}(0)$. Then there are a $G$-invariant metric on $\gstarM$ and a $p_E \in \N$ such that if $p \geq p_E$, then
\[
\bigl(\ker_{L^2_T} D^E_{p, 1}\bigr)^G
\]
is finite-dimensional, and one has
\[
\ind^G_{L^2_T} D^E_{p, 1} = \ind D_{M_0}^{E_0} = \int_{M_0} \ch(E_0) e^{\frac{p}{2}c_1(L_0)} \hat A(M_0).
\]
Here $E_0 := (E|_{(\munab)^{-1}(0)})/G$ and $L_0 := (L|_{(\munab)^{-1}(0)})/G$.
\end{theorem}
In the compact case, results analogous to Theorem \ref{thm index} were obtained in \cite{PV3, TZ98b}.
Theorem \ref{thm index} will be proved in Section \ref{sec twisted}. Some applications are given in Subsection \ref{sec appl ex twist}.


\section{The square of the deformed Dirac operator} \label{sec Bochner}

We now turn to proving Theorems \ref{thm quant well defd} and \ref{thm [Q,R]=0 large p}. As in \cite{HM, TZ98}, the starting point is an explicit formula, given in Theorem \ref{thm Bochner}, for the square of the deformed Dirac operator $D_{p, t}$ of Definition \ref{def deformed Dirac}. This is the basis of the localisation estimates, Propositions \ref{prop loc V} and \ref{prop loc U}, that will be used to prove Theorems \ref{thm quant well defd} and \ref{thm [Q,R]=0 large p}.

We continue using the notation of Section \ref{sec Dirac red} and Subsection \ref{sec non-cocpt}. We will also write $d_M$ and $d_G$ for the dimensions of $M$ and $G$, respectively. We denote  the Riemannian metric on $M$ induced by the given $\Spinc$-structure by $(\relbar, \relbar)$. The associated Levi--Civita connection on $TM$ will be denoted by $\nabla^{TM}$.

\subsection{A Bochner formula}

Let us fix some notation that will be used in the expression for $D_{p, t}^2$.
Let $\{h_1, \ldots, h_{d_G}\}$ be an orthonormal frame for $\gstarM$ with respect to a given $G$-invariant metric. (Such a frame can be obtained for example by applying the Gram-Schmidt procedure to a constant frame.) Let $\{h_1^*, \ldots, h_{d_G}^*\}$ be the dual frame of $M \times \kg \to M$. Let $\munab_1, \ldots, \munab_{d_G} \in C^{\infty}(M)$ be the functions such that
\begin{equation} \label{eq def muj}
\munab = \sum_{j=1}^{d_G} \mu_j^{\nabL} h_j,
\end{equation}
so that
\[
(\munab)^* = \sum_{j=1}^{d_G} \mu_j^{\nabL} h_j^*,
\]
and
\begin{equation} \label{eq decomp v}
{\vnab} = 2\sum_{j=1}^{d_G} \mu_j^{\nabL} V_j,
\end{equation}
where $V_j$ is the vector field given by
\begin{equation} \label{eq def Vj}
V_j(m) = \bigl(h_j^*(m) \bigr)^M_m,
\end{equation}
at a point $m \in M$. Consider the norm-squared function $\cH^{\nabla}$ of $\mu^{\nabla}$, given by
\begin{equation} \label{eq def H}
\cH^{\nabla}(m) = \| \mu^{\nabla}(m) \|_m^2 = \sum_{j=1}^{d_G} \mu_j^{\nabla}(m)^2.
\end{equation}
Here $\|\cdot\|_m$ is the norm on $\kg^*$ induced by $(\relbar, \relbar)_m$. 

We will use the operators $\cL^{\cS_p}_{h_j^*}$ on $\Gamma^{\infty}(\cS_p)$ given by
\[
\bigl(\cL^{\cS_p}_{h_j^*} s\bigr)(m) = \bigl(\cL^{\cS_p}_{h_j^*(m)}s\bigr)(m).
\]
Finally,  for any vector field $u$ on $M$, consider the commutator vector field $[u, (h^*_j)^M ]$, given by
\[
[u, (h^*_j)^M](m) = \bigl[u, h^*_j(m)^M \bigr](m).
\]
Here $h^*_j(m)^M$ is the vector field induced by $h^*_j(m) \in \kg$, and $[\relbar, \relbar]$ is the Lie bracket of vector fields. Importantly, for fixed $m$, the vector fields $V_j$ and $h^*_j(m)^M$ are equal at the point $m$, but not necessarily at other points.

The square of $D_{p, t}$ has the following form.
\begin{theorem} \label{thm Bochner}
One has
\[
D_{p, t}^2 = D_p^2 + tA + (2p + 1)2\pi t \cH^{\nabla} + \frac{t^2}{4}\|{\vnab}\|^2 -2 i t \sum_{j=1}^{d_G} \munab_j \cL^{\cS_p}_{h_j^*},
\]
where $A$ is a vector bundle endomorphism of $\cS_{p}$, given in terms of a local orthonormal frame $\{e_1, \ldots, e_{d_M}\}$ of $TM$ by
\begin{multline} \label{eq def A}
A := \frac{i}{4}\sum_{k = 1}^{d_M} c(e_k)c\bigl(\nabla^{TM}_{e_k} {\vnab} \bigr) 
+ \frac{i}{2}\sum_{j=1}^{d_G} c(\grad \mu^{\nabla}_j) c(V_j) \\
	-\frac{i}{2} \sum_{j = 1}^{d_G} \sum_{k=1}^{d_M} \munab_j c(e_k)c\bigl( [e_k, (h_j^*)^M - V_j] \bigr).
\end{multline}
\end{theorem}

\subsection{Lie derivatives of spinors}

An important ingredient of the proof of Theorem \ref{thm Bochner} is an expression for the Lie derivative of sections of $\cS_p$.
\begin{lemma} \label{lem Lie spinors}
Let  $X \in \kg$. Then, as operators on $\Gamma^{\infty}(\cS_p)$, one has
\[
\cL^{\cS_p}_X = \nabla^{\cS_p}_{X^M} - B_X - (2p+1)\pi i \munab_X,
\]
where, in terms of a local orthonormal frame $\{e_1, \ldots, e_{d_M}\}$ of $TM$,
\[
B_X := \frac{1}{4}\sum_{k, l = 1}^{d_M} \bigl(\nabla_{e_k}X^M, e_l \bigr)c(e_k)c(e_l).
\]
\end{lemma}
\begin{proof}
Let $X \in \kg$ be given.
We give a local argument on an open subset $U \subset M$, using the decomposition \eqref{eq decomp Sp} of $\cS_P|_U$. 
Let $ \nabla^{L|_U^{1/2}}$ be the connection on $L|_U^{1/2} \to U$ induced by $\nabla$. We first note that
\begin{equation} \label{eq Lie nabla Lhalf}
\cL^{L|_U^{1/2}}_X = \nabla^{L|_U^{1/2}}_{X^M} - i\pi\munab_X|_U.
\end{equation}
Indeed, if $t_1, t_2 \in \Gamma^{\infty}\bigl( L|_U^{1/2} \bigr)$, then by definition of $\munab$,
\begin{multline*}
\bigl(\cL^{L|_U^{1/2}} t_1 \bigr) \otimes t_2 + t_1 \otimes \bigl(\cL^{L|_U^{1/2}} t_2 \bigr) 
= \cL^{L|_U}_X(t_1 \otimes t_2) \\
	= \bigl( \nabL_{X^M} - 2\pi i\munab_X \bigr)(t_1 \otimes t_2) \\
	=  \left(\bigl( \nabla^{L|_U^{1/2}}_{X^M} - i\pi\munab_X \bigr)t_1 \right)\otimes t_2 + t_1 \otimes \left(\bigl( \nabla^{L|_U^{1/2}}_{X^M} - i\pi\munab_X \bigr)t_2 \right).
\end{multline*}

Let $s \in \Gamma^{\infty}(\cS^{\Spin}_U)$. Then
\begin{equation} \label{eq Lie nabla spin}
\cL^{\cS^{\Spin}_U}_X s = \nabla^{\cS^{\Spin}_U}_{X^M}s - B_X s.
\end{equation}
%
Let $t_1, \ldots, t_{2p+1} \in \Gamma^{\infty}\bigl(L|_U^{1/2} \bigr)$. Then
\[
s \otimes t_1 \otimes \cdots \otimes t_{2p+1} \in \Gamma^{\infty}\bigl(\cS^{\Spin}_U \otimes L|_U^{p+1/2}\bigr) = \Gamma^{\infty}(\cS_p|_U). 
\]
Because of \eqref{eq Lie nabla Lhalf} and \eqref{eq Lie nabla spin}, one has
\begin{multline*}
\cL^{\cS_p}_X (s \otimes t_1 \otimes \cdots \otimes t_{2p+1}) = \\ 
	\bigl(\cL^{\cS^{\Spin}_U}_Xs\bigr)  \otimes t_1 \otimes \cdots \otimes t_{2p+1} +
		 s\otimes \left(\sum_{j=1}^{2p+1} t_1 \otimes \cdots \otimes \bigl(\cL^{L|_U^{1/2}}_X t_j \bigr) \otimes \cdots \otimes t_{2p+1} \right) =  \\
	\bigl(\nabla^{\cS^{\Spin}_U}_{X^M}s\bigr)  \otimes t_1 \otimes \cdots \otimes t_{2p+1} +
		 s\otimes \left(\sum_{j=1}^{2p+1} t_1 \otimes \cdots \otimes \bigl(\nabla^{L|_U^{1/2}}_{X^M} t_j  \bigr)\otimes \cdots \otimes t_{2p			+1}  \right) \\
		 -\bigl(B_X + (2p+1)\pi i \mu_X \bigr) s \otimes t_1 \otimes \cdots \otimes t_{2p+1}  = \\
 \left(\nabla^{\cS_p}_{X^M} - B_X - (2p+1) \pi i  \munab_X) \right) s \otimes t_1 \otimes \cdots \otimes t_{2p+1}.
\end{multline*}
\end{proof}

\subsection{Proof of the Bochner formula}

Using Lemma \ref{lem Lie spinors}, we can prove Theorem \ref{thm Bochner}.

As in the equality (1.26) in \cite{TZ98}, the fact that $\nabla^{\cS_p}$ satisfies a Leibniz rule with respect to the Clifford action (see e.g.\ Proposition 4.11 in \cite{LM}) implies that
\begin{equation} \label{eq Bochner 0}
 D_{p, t}^2 = D_p^2 + \frac{it}{2} \Sk c(e_k) c(\nabla^{TM}_{e_k}{\vnab})
	-it \nabla^{\cS_p}_{\vnab} + \frac{t^2}{4} \|{\vnab}\|^2.
\end{equation}
The main part of the proof of Theorem \ref{thm Bochner} is a computation of an expression for the first-order term $\nabla^{\cS_p}_{\vnab}$. 

By \eqref{eq decomp v}, we have
\[
\nabla^{\cS_p}_{{\vnab}} = 2 \Sj \mu^{\nabla}_j \nabla^{\cS_p}_{V_j}.
\]
By Lemma \ref{lem Lie spinors}, one has for all $s \in \Gamma^{\infty}(\cS_p)$, all $m \in M$ and all $j$,
\begin{multline*}
\bigl( \nabla^{\cS_p}_{V_j} s \bigr)(m) = \bigl( \nabla^{\cS_p}_{h_j^*(m)^M} s \bigr)(m) \\
=\left(\left( \cL^{\cS_p}_{h_j^*(m)} + B_{h_j^*(m)} + (2p+1)\pi i \munab_j \right)s\right)(m).
\end{multline*}
Multiplying this identity by $2\mu^{\nabla}_j(m)$ and summing over $j$, we obtain
\begin{multline} \label{eq Bochner 1}
\bigl( \nabla^{\cS_p}_{\vnab} s \bigr)(m) = \left(\biggl(2\sum_{j=1}^{d_G} \munab_j \cL^{\cS_p}_{h_j^*}\biggr)s\right)(m) \\
	+\left(\biggl(2\Sj \munab_j  B_{h_j^*(m)}\biggr)s\right)(m) + \bigl((2p+1)2\pi i \cH^{\nabL} s\bigr)(m).
\end{multline}

Lemma B.2 in \cite{HM} allows us to compute
\begin{multline*}
\left(\biggl(2\Sj \munab_j  B_{h_j^*(m)}\biggr)s\right)(m) = 
\frac{1}{2}\Sj \mu^{\nabla}_j \sum_{k, l = 1}^{d_M} \bigl(\nabla_{e_k}h_j^*(m)^M, e_l \bigr)c(e_k)c(e_l) \\
= \left(\biggl( 
\frac{1}{4}\Sk c(e_k)c\bigl( \nabla^{TM}_{e_k}{\vnab} \bigr) - \frac{1}{2}\Sj c(\grad \mu^{\nabla}_j)c(V_j) \biggr. \right. \\
\hspace{3.5cm} \left. \biggl.
+\frac{1}{2}\Sj \Sk \mu_j^{\nabla} c(e_k) c\bigl( [e_k, (h_j^*)^M - V_j] \bigr)
\biggr)s\right)(m) \\
= i\left( \biggl( A - \frac{it}{2} \Sk c(e_k) c\bigl( \nabla^{TM}_{e_k} {\vnab}\bigr)  \biggr)s\right)(m).
\end{multline*}
Theorem \ref{thm Bochner} follows from this equality and \eqref{eq Bochner 0} and \eqref{eq Bochner 1}.

\begin{remark}
Lemma B.3 in \cite{HM} does not apply in the general $\Spinc$-case, so that $\grad \mu^{\nabla}_j$, which appears in the expression for the operator $A$, cannot be worked out further in the present setting.
\end{remark}

\subsection{An estimate for the operator $A$}

To prepare for the localisation estimates in Section \ref{sec loc est}, we show that the operator $A$ in Theorem \ref{thm Bochner} satisfies a certain estimate with respect to a rescaling of the metric on $\gstarM$ by a function.

For any positive, $G$-invariant smooth function $\psi \in C^{\infty}(M)^G$, consider the metric 
\begin{equation} \label{eq metric psi}
\{\psi(m) (\relbar, \relbar)_m\}_{m \in M} 
\end{equation}
on $\gstarM$. Let $A^{\psi}$ be the operator in Theorem \ref{thm Bochner}, defined with respect to this metric. 
In the choice of the metric on $\gstarM$ in Proposition \ref{prop choice metric}, we will use the following property of the dependence of the operator $A^{\psi}$ on $\psi$.
\begin{lemma} \label{lem est A}
There are $G$-invariant, positive, continuous functions $F_1, F_2 \in C(M)^G$ such that for all $G$-invariant, positive smooth functions $\psi \in C^{\infty}(M)$, one has the pointwise estimate
\begin{equation} \label{eq est A}
\|A^{\psi}\| \leq F_1 \psi + F_2 \|d\psi\|.
\end{equation}
\end{lemma}
\begin{proof}
Let $\psi \in C^{\infty}(M)^G$ be a $G$-invariant, positive smooth function. With respect to the metric \eqref{eq metric psi} rescaled by $\psi$, we use the orthonormal frame of $\gstarM$ made up of the functions
\[
h_j^{\psi} := \frac{1}{\psi^{1/2}}h_j.
\]
The dual frame of $M \times \kg \to M$ consists of the functions
\[
(h_j^{\psi})^* =  \psi^{1/2} h_j^*.
\]
Let $(\mu_j^{\nabla})^{\psi}$ be defined like the functions $\mu_j^{\nabla}$ in \eqref{eq def muj}, with $h_j$ replaced by $h_j^{\psi}$. Analogously, let $V_j^{\psi}$ be the vector field defined like $V_j$ in \eqref{eq def Vj}, with the same replacement. Then
\begin{equation} \label{eq mu V psi}
\begin{split}
(\mu_j^{\nabla})^{\psi} &= {\psi^{1/2}} \mu_j^{\nabla}; \\
V_j^{\psi} &= \psi^{1/2}V_j.
\end{split}
\end{equation}
It follows for example from the latter two equalities and \eqref{eq decomp v} that the vector field $({\vnab})^{\psi}$, defined like ${\vnab}$ with the metric on $\gstarM$ rescaled by $\psi$, equals
\begin{equation} \label{eq v psi}
({\vnab})^{\psi} = \psi \vnab.
\end{equation}

We start with some local computations for each term in the definition \eqref{eq def A} of the operator $A^{\psi}$. Let $\{e_1, \ldots, e_{d_M}\}$ be a local orthonormal frame for $TM$. By \eqref{eq v psi}, we have for all $k$,
\[
\nabla^{TM}_{e_k}({\vnab})^{\psi} = \psi \nabla^{TM}_{e_k}{\vnab} + e_k(\psi) \vnab. 
\]
Hence
\[
\begin{split}
\left\|  \frac{i}{4}\sum_{k = 1}^{d_M} c(e_k)c\bigl(\nabla^{TM}_{e_k} ({\vnab})^{\psi} \bigr)  \right\| 
	&\leq \frac{1}{4} \Sk \bigl(\psi \|\nabla^{TM}_{e_k} {\vnab}\|  +  \|e_k(\psi)\| \|{\vnab}\|  \bigr) \\
	& \leq a_1 \psi + a_2 \|d\psi\|, 
\end{split}
\]
with 
\[
\begin{split}
a_1 &:= \frac{1}{4}\Sk \|\nabla^{TM}_{e_k} {\vnab}\|;\\
a_2 &:= \frac{1}{4}d_M \|{\vnab}\|.
\end{split}
\]
Note that the function $a_1$ is not defined globally, and is not $G$-invariant on its domain in general. We will come back to this later.

Secondly, because of \eqref{eq mu V psi}, we have
\begin{multline} \label{eq est A 1}
\left\| \frac{i}{2}\sum_{j=1}^{d_G} c\bigl(\grad (\mu^{\nabla}_j)^{\psi} \bigr) c(V_j^{\psi}) \right\| \\
	\leq \frac{1}{2} \Sj \left( \psi \|\grad \mu_j^{\nabla}\|\,  \|V_j\| + |\mu_j^{\nabla}| \, \psi^{1/2} \|\grad \psi^{1/2}\| \, \|V_j\| \right).
\end{multline}
Since $\psi^{1/2} \|\grad \psi^{1/2}\| = \frac{1}{2}\|d\psi\|$, \eqref{eq est A 1} is at most equal to
\[
b_1 \psi + b_2 \|d\psi\|,
\]
with
\[
\begin{split}
b_1 &:=  \frac{1}{2} \Sj \|\grad \mu_j^{\nabla}\| \,  \|V_j\|;\\
b_2 &:= \frac{1}{4} \Sj |\mu_j^{\nabla}| \, \|V_j\|.
\end{split}
\]

Finally, Lemma C.8 in \cite{HM} implies that
\[
\bigl[e_k, \bigl( (h_j^*)^{\psi} \bigr)^M - V_j^{\psi} \bigr] = \psi^{1/2} [e_k, (h_j^*)^M - V_j] -e_k(\psi^{1/2})V_j.
\]
Therefore,
\begin{multline} \label{eq est A 2}
\left\| -\frac{i}{2} \sum_{j = 1}^{d_G} \sum_{k=1}^{d_M} (\munab_j)^{\psi} c(e_k)c\bigl( \bigl[e_k, \bigl( (h_j^*)^{\psi} \bigr)^M - V_j^{\psi}\bigr] \bigr) \right\| \\
	\leq \frac{1}{2} \Sj \Sk \left( 
	\psi |\munab_j| \, \bigl\|  [e_k, (h_j^*)^M - V_j]  \bigr\| + \psi^{1/2} \|e_k(\psi^{1/2})\| \, |\mu^{\nabla}_j| \, \|V_j\|.
	\right)
\end{multline}
Since
\[
 \psi^{1/2} \|e_k(\psi^{1/2})\| = \frac{1}{2}\|e_k(\psi)\| \leq \frac{1}{2} \|d\psi\|,
\]
we find that \eqref{eq est A 2} is at most equal to 
\[
c_1 \psi + c_2 \|d\psi\|,
\]
with 
\[
\begin{split}
c_1 &:=  \frac{1}{2} \Sj \Sk  |\munab_j| \, \bigl\|  [e_k, (h_j^*)^M - V_j]  \bigr\|  ;\\
c_2 &:=   \frac{d_M}{2} \Sj |\munab_j| \, \|V_j\|.
\end{split}
\]

The functions $a_j$, $b_j$ and $c_j$ are not all defined globally and/or $G$-invariant. To get a global estimate for $A$, let $W\subset M$ be an open subset that intersects all $G$-orbits in nonempty, relatively compact sets. By Lemmas C.1 and C.2 in \cite{HM}, there are $G$-invariant, positive, continuous functions $F_1$ and $F_2$ on $M$, and local orthonormal frames of $TM$ around each point in $W$, such that on $W$, with respect to these frames, one has
\[
\begin{split}
a_1 + b_1 + c_1 &\leq F_1; \\
a_2 + b_2 + c_2 & \leq F_2.
\end{split}
\]
Then the estimate \eqref{eq est A} holds on $W$. Since both sides of  \eqref{eq est A} are $G$-invariant, and the definition of $A$ is independent of the local orthonormal frame chosen, we get the desired estimate on all of  $M$.
\end{proof}


\section{Localisation estimates} \label{sec loc est}

Two localisation estimates are at the cores of the proofs of Theorems \ref{thm quant well defd} and \ref{thm [Q,R]=0 large p}. These are Propositions \ref{prop loc V} and \ref{prop loc U} below. In the proofs of these estimates, we will not use the assumption that $0$ is a $\Spinc$-regular value of $\munab$. They therefore also hold in the singular case. The regularity assumption is only needed to apply the arguments near $\munabinv(0)$ to obtain Theorem \ref{thm [Q,R]=0 large p}. 

The localisation estimates are stated in terms of certain Sobolev norms.

\subsection{Sobolev norms and estimates for $D_{p, t}$}

Theorem \ref{thm quant well defd} follows from the fact that for large $t$, the operator $D_{p, t}$ induces a Fredholm operator between certain Sobolev spaces. By an elliptic regularity argument, the index of this operator is precisely the $G$-invariant transversally $L^2$-index $\ind^G_{L^2_T}$ of $D_{p, t}$. These Sobolev spaces and the index theory on them that we will use, were introduced in Section 4 of \cite{HM}. We will not need to go into the details of these spaces, but will refer to the relevant results in \cite{HM}. We do need certain ingredients of the definition of these spaces. 

One of these is a smooth cutoff function $f$ on $M$ (see Definition \ref{def cutoff fn}). We will also consider \emph{transversally compactly supported} sections of vector bundles, by which we mean sections whose support is mapped to a compact set by the quotient map $M \to M/G$. Let $\Gamma^{\infty}_{tc}(\cS_p)^G$ be the space of $G$-invariant, smooth, transversally compactly supported sections of $\cS_p$. For $k \in \N$, and $s, s' \in \Gamma^{\infty}_{tc}(\cS_p)^G$, we set
\begin{equation} \label{eq def Sob norm}
(fs, fs')_k := \sum_{j=0}^k (fD_{p}^j s, fD_p^j s')_{L^2(\cS_p)}.
\end{equation}
(Note that $fD_{p}^j s$ and $fD_{p}^j s'$ are compactly supported for all $j$.) By Lemma 4.4 in \cite{HM}, this inner product is independent of $f$, since $s$ and $s'$ are $G$-invariant. We will write $\|\cdot\|_k$ for the induced norm on $f\Gamma^{\infty}_{tc}(\cS_p)^G$. 

These Sobolev norms allow us to state the localisation estimates we will use. Fix a  $G$-invariant  open neighbourhood $V$ of the set $\Crit({\vnab})$ of zeros of ${\vnab}$. We assumed that $\Crit({\vnab})$ is cocompact, so we may assume that $V$ is relatively cocompact, in the sense that $V/G$ is a relatively compact subset of $M/G$.
\begin{proposition} \label{prop loc V}
There is a $G$-invariant metric on $\gstarM$, and there are $t_0, C, b > 0$, such that for all $t \geq t_0$, all $p \in \N$, and all $G$-invariant $s \in \Gamma^{\infty}_{tc}(\cS_p)^G$ with support disjoint from $V$, one has
\begin{equation} \label{eq est loc}
\|f D_{p, t}s\|_0^2 \geq C \bigl( \|fs\|_1^2 + (t-b) \|fs\|_0^2\bigr).
\end{equation}
\end{proposition}
\begin{proposition} \label{prop loc U}
The metric on $\gstarM$ used in Proposition \ref{prop loc V} can be chosen such that, in addition to the conclusions of that proposition, for every $G$-invariant open neighbourhood $U$ of $\munabinv(0)$, there are $p_0 \in \N$ and $t_0, C, b > 0$, such that for all $t \geq t_0$ and $p \geq p_0$, and all $G$-invariant $s \in \Gamma^{\infty}_{tc}(\cS_p)^G$ with support disjoint from $U$, the estimate \eqref{eq est loc} holds.
\end{proposition}
So the estimate holds for all $s$ supported outside $V$ for all $p$, and for all $s$ supported outside the smaller set $U$ for large $p$.

It is important that the metric on $\gstarM$ used in Propositions \ref{prop loc V} and \ref{prop loc U} is the same. They therefore actually form one result, with two conclusions.

In addition, note that the condition that $t \geq t_0$ can be absorbed into the choice of the metric on $\gstarM$, since multiplying this metric by a constant results on multiplying the vector field $v_{\nabla}$ by the same constant. The parameter $t$ was just introduced to make the arguments that follow clearer.

\subsection{Choosing the metric on $M \times \kg^*$}

One advantage of using a family of inner products on $\kg^*$, i.e.\ a metric on $\gstarM$, is that this allows us to define the $G$-invariant vector field ${\vnab}$ and the $G$-invariant function $\cH^{\nabla}$. Another advantage that is very important for our arguments is that choosing this metric in a suitable way allows us to control the terms that appear in the Bochner formula in Theorem \ref{thm Bochner}. 

To make this precise, consider the $G$-invariant, positive, continuous function $\eta$ on $M$ defined by\footnote{What follows holds for any $G$-invariant, positive, continuous function $\eta$.}
\begin{equation} \label{eq eta}
\eta(m) =\int_G f(gm)\|df\|(gm) \, dg,
\end{equation}
for $m \in M$. 
\begin{proposition} \label{prop choice metric}
The $G$-invariant metric on the bundle $\gstarM$ can be chosen in such a way that 
 for all $m \in M \setminus V$,
\begin{align}
\cH^{\nabla}(m) &\geq 1; \label{eq est H} \\
\|\vnab_m\| &\geq 1 + \eta(m), \label{eq est X1H}
\end{align}
and 
there is a positive constant $C$, such that
for all $m \in M$, the operator $A_m$ on $(\cS_p)_m$ is bounded below by
\begin{equation} \label{eq A bdd below}
A_m \geq -\|\vnab_m\|^2 -C.
\end{equation}
\end{proposition}
\begin{proof}
Fix any $G$-invariant metric $\{(\relbar, \relbar)_m\}_{m \in M}$ on $\gstarM$. Let the $G$-invariant, positive, continuous functions $F_1$ and $F_2$ be as in Lemma \ref{lem est A}. Set
\[
\begin{split}
\varphi_1&:= \min \left(\cH^{\nabla},  \frac{\|{\vnab}\|}{1+\eta}, \frac{\|{\vnab}\|^2}{2F_1} \right) \\
\varphi_2 &:= \frac{\|{\vnab}\|^2}{2F_2}.
\end{split}
\]
This defines  $G$-invariant, continuous functions $\varphi_1$ and $\varphi_2$ on $M$, which are positive outside $\Crit({\vnab})$. Since $\Crit({\vnab})/G$ is compact, the functions $\varphi_j$ have uniform lower bounds outside the neighbourhood $V$ of $\Crit({\vnab})$. Hence there are positive, $G$-invariant, continuous functions $\tilde \varphi_j$ on $M$, such that
\[
\tilde \varphi_j|_{M \setminus V} = \varphi_j|_{M \setminus V},
\] 
for $j = 1, 2$. By Lemma C.3 in \cite{HM}, there is a $G$-invariant, positive, smooth function $\psi$ on $M$, such that 
\[
\begin{split}
\psi^{-1} & \leq \tilde \varphi_1; \\
\|d(\psi^{-1})\|& \leq \tilde \varphi_2.
\end{split}
\]
Consider the metric $\{\psi(m)(\relbar, \relbar)_m\}_{m \in M}$ on $\gstarM$, obtained by rescaling the given metric by $\psi$. We claim that this metric has the desired properties.

First of all, the function $\cH^{\nabla}_{\psi}$ and the vector field $({\vnab})^{\psi}$ associated to this metric satisfy, outside $V$,
\[
\begin{split}
\cH^{\nabla}_{\psi} &= \psi \cH^{\nabla} \geq \varphi_1^{-1} \cH^{\nabla} \geq 1; \\
\|({\vnab})^{\psi}\| &= \psi \|{\vnab}\| \geq \varphi_1^{-1} \|{\vnab}\| \geq 1+\eta.
\end{split}
\]

Furthermore, by Lemma \ref{lem est A}, the operator $A^{\psi}$ in Theorem \ref{thm Bochner}, associated to the metric on $\gstarM$ rescaled by $\psi$, satisfies, outside $V$,
\[
\begin{split}
\frac{\|A^{\psi}\|}{\|({\vnab})^{\psi}\|^2} &\leq \frac{F_1 \psi + F_2 \|d\psi\|}{\psi^2 \|{\vnab}\|^2} \\
&= \frac{F_1}{\|{\vnab}\|^2} \psi^{-1} + \frac{F_2}{\|{\vnab}\|^2} \|d(\psi^{-1})\| \\
&\leq 1.
\end{split}
\]
Hence $\|A^{\psi}\| \leq \|({\vnab})^{\psi}\|^2$, on $M \setminus V$. Since $V$ is relatively cocompact and $A^{\psi}$ is $G$-equivarant, it is bounded on $V$. So
\[
A^{\psi} \geq -C
\]
on $V$, for a certain $C>0$. We conclude that
\[
A^{\psi} \geq -\|({\vnab})^{\psi}\|^2 - C
\]
on all of $M$. 
\end{proof}

\begin{remark}
A priori, the choice of metric on $\gstarM$ could influence $\ind^G_{L^2_T}(D_{p, t})$, if $\Crit(\vnab)$ changes (while staying cocompact). Multiplying a metric by a function $\psi$ as in Proposition \ref{prop choice metric} does not change $\Crit(\vnab)$, however, and the second point in Theorem 2.15 in \cite{Braverman2} implies that $\ind^G_{L^2_T}(D_{p, t})$ is independent of $\psi$. It follows from Theorem \ref{thm [Q,R]=0 large p} that this index is independent of the metric in general, as long as $\Crit(\vnab)$ is cocompact, for large enough $p$.

Also note that one may take $t_0 = 1$ in Theorem \ref{thm quant well defd}, since, in the notation of the proof of Proposition \ref{prop choice metric},
\[
\frac{it}{2}c\bigl(  (\vnab)^{\psi} \bigr) = \frac{i}{2}c\bigl(  (\vnab)^{t\psi} \bigr).
\]

\end{remark}

\subsection{Proofs of the localisation estimates}

Proposition \ref{prop choice metric} allows us to prove Propositions \ref{prop loc V} and \ref{prop loc U}. Fix a $G$-invariant metric on $\gstarM$ as in Proposition \ref{prop choice metric}, and a smooth cutoff function $f$. It will be useful to consider the operator
\[
\tilDpt: f\Gamma^{\infty}_{tc}(\cS_p)^G \to f\Gamma^{\infty}_{tc}(\cS_p)^G,
\]
defined by
\begin{equation} \label{eq def D tilde}
\tilDpt fs = f D_{p, t}s,
\end{equation}
for $s \in \Gamma^{\infty}_{tc}(\cS_p)^G$. We will write $\widetilde{D}_p := \widetilde{D}_{p, 0}$.

We need some arguments to account for the fact that, unlike $D_{p, t}$, the operator $\tilDpt$ is not symmetric with respect to the $L^2$-inner product. Let $\tilDpt^*$ be its formal adjoint. 
Combining Theorem \ref{thm Bochner} and Proposition \ref{prop choice metric}, one obtains the following key estimate for the operator $\tilDpt^* \tilDpt$.
\begin{corollary} \label{cor est adjoint}
One has
\[
\tilDpt^* \tilDpt = \widetilde{D_p}^* \widetilde{D_p} + tB + (2p+1)2\pi t \cH^{\nabL} + \frac{t^2}{4} \|{\vnab}\|^2,
\]
where $B$ is a vector bundle endomorphism of $\cS_p$ for which there is a constant $C > 0$ such that one has the pointwise estimate
\[
B \geq -C \bigl( \|{\vnab}\|^2 +1\bigr).
\]
\end{corollary}
\begin{proof}
This was proved in the symplectic setting in Proposition 6.7 in \cite{HM}.
The arguments remain the same, however. 
References to Theorem 5.1 and to Proposition 6.6 in the proof of Proposition 6.7 in \cite{HM} should be replaced by references to Theorem \ref{thm Bochner} and Proposition \ref{prop choice metric} in the present paper, respectively. Note that the last term in the Bochner formula of Theorem \ref{thm Bochner} vanishes on $G$-invariant sections.
\end{proof}

The proofs of Propositions \ref{prop loc V} and \ref{prop loc U} are now the same as the proofs of Propositions 6.1 and 6.3 in \cite{HM}, with Corollary \ref{cor est adjoint} playing the role of Proposition 6.7 in \cite{HM}.

\subsection{Proofs of Theorems \ref{thm quant well defd} and \ref{thm [Q,R]=0 large p}} \label{sec proofs theorems}

Theorem \ref{thm quant well defd} follows from Proposition \ref{prop loc V}, in the way that Theorem 3.4 in \cite{HM} follows from Proposition 6.1 in \cite{HM}. Indeed, for $t \geq b+1$ and any $p$ in Proposition \ref{prop loc V}, one has 
\[
\|fD_{p, t}\|_0^2 \geq C \|fs\|_0^2,
\]
for $G$-invariant sections $s \in \Gamma^{\infty}_{tc}(\cS_p)^G$ with support disjoint from the set $V$. By Proposition 4.7 in \cite{HM}, the operator $\tilDpt$ therefore extends to a Fredholm operator between Sobolev spaces. By Proposition 4.8 in \cite{HM},
 $\ker_{L^2_T}(D_{p,t})^G$ is finite-dimensional, and 
 the index of the Fredholm operator induced by $\tilDpt$ equals $\ind^G_{L^2_T}(D_{p, t}$) .  It is noted in part 2 of Theorem 2.15 in \cite{Braverman2} that this index is independent of $t$, so that Theorem \ref{thm quant well defd} follows.

To prove Theorem \ref{thm [Q,R]=0 large p}, we apply Proposition \ref{prop loc U}. This proposition shows that the arguments in Sections 6.5 and 7 of \cite{HM} apply to the operator $D_{p, t}$, for large enough $p$ and $t$. Therefore, the techniques from Sections 8 and 9 in \cite{BL91} can be used as in \cite{HM, MZ, TZ98}.
It follows that, for large enough $p$ and $t$,
\begin{equation} \label{eq [Q,R]=0 p}
\ind^G_{L^2_T}(D_{p, t}) = \ind(D^{\nabla^{0}}_{M_0}),
\end{equation}
where $D^{\nabla^0}_{M_0}$ is the $\Spinc$-Dirac operator on the reduced space $M_0$ associated to the $\Spinc$-structure of Lemma \ref{lem Spinc normal}, and the connection $\nabla^{0}$ on the line bundle $L_0^{2p+1} \to M_0$ induced by the connection $\nabL$ on $L$. Theorem \ref{thm [Q,R]=0 large p} therefore follows from the proposition below.
\begin{proposition} \label{prop Spinc p}
For all $p \in \N$, there exists a $G$-equivariant $\Spinc$-structure on $M$, and a connection on the associated determinant line bundle, such that the corresponding invariant $\Spinc$-quantisation is
\[
Q^{\Spinc}(M)^G = \ind^G_{L^2_T}(D_{p, t}),
\]
for $t$ large enough, and 
\[
Q^{\Spinc}(M_0 ) = \ind(D^{\nabla^{0}}_{M_0}).
\]
\end{proposition}
\begin{proof} 
Let $P \to M$ be the given $G$-equivariant principal $\Spinc$-structure  on $M$. Let $P' \to M$ be the $G$-equivariant  $\Spinc$-structure with determinant line bundle $L' = L^{2p+1}$. Explicitly, 
\[
P' = P \times_{\U(1)} \UF(L^p),
\]
where $\UF$ denotes the unitary frame bundle. (See e.g.\ part (2) of Proposition D.43 in \cite{GGK}.) Let $\nabla'$ be the connection on $L'$ induced by $\nabla$. 

Let $\cS'\to M$ be the spinor bundle associated to $P'$. Then $\cS' = \cS_p$ (see e.g.\ (D.15) in \cite{LM}). Hence the connection $\nabla^{\cS'}$ on $\cS'$ induced by $\nabla'$ and the Levi--Civita connection on $TM$ equals the connection on $\cS_p$ used to define the Dirac operator $D_p$. Therefore, the $\Spinc$-Dirac operator $D'$ on $\cS'$ equals the operator $D_p$. Furthermore, the $\Spinc$-momentum map $\mu^{\nabla'}: M \to \kg^*$ associated to $\nabla'$ is given by
\begin{equation} \label{eq mu p}
2\pi i \mu^{\nabla'}_X =  \nabla'_{X^M} - \cL^{L^{2p+1}}_X = 2\pi i (2p+1)\munab_X,
\end{equation}
for all $X \in \kg$. It follows that the induced vector field $v^{\nabla'}$ equals
\[
v^{\nabla'} = (2p+1)\vnab.
\]
We conclude that the deformed Dirac operator on $\cS'$ associated to $\nabla'$ is
\[
D'_{1, t} = D' + \frac{it}{2}c(v^{\nabla'}) = D_p + \frac{(2p+1)it}{2}c({\vnab}) = D_{p, (2p+1)t}. 
\]

Let $t_0, t'_0 \in \R$ be as in Theorem \ref{thm quant well defd}, for he operators $D_{p, t}$ and $D'_{p, t}$, respectively. 
This theorem states that $\ind^G_{L^2_T}(D_{p, t})$ does not depend on $t \geq t_0$.
Hence, if  
\[
\begin{split}
t &\geq t_0; \\
t' &\geq t'_0; \text{ and} \\
(2p+1)t' &\geq t_0,
\end{split}
\]
then, with respect to the $\Spinc$-structure $P'$ and the connection $\nabla'$,
\[
Q^{\Spinc}(M)^G = \ind^G_{L^2_T}(D'_{1, t'}) = \ind^G_{L^2_T}(D_{p, (2p+1)t'})  =  \ind^G_{L^2_T}(D_{p, t}).
\]

Finally, by \eqref{eq mu p}, one has 
\[
M_0 = (\mu^{\nabla'})^{-1}(0)/G = (\mu^{\nabla})^{-1}(0)/G.
\]
And the connection $(\nabla')^0$ on $L'_0 = L_0^{2p+1}$ is the one  induced by the connection $\nabla^0$ on $L_0$, so the second claim follows as well.
\end{proof}


\section{Twisted $\Spinc$-Dirac operators} \label{sec twisted}

Theorem \ref{thm index} can be proved by generalising the steps in the proofs of Theorems \ref{thm quant well defd} and \ref{thm [Q,R]=0 large p} to twisted $\Spinc$-Dirac operators. 

\subsection{A Bochner formula for twisted Dirac operators}

As in the case for untwisted Dirac operators, the proof of Theorem \ref{thm index} starts with an expression for the square of the deformed Dirac operator $D^E_{p, t}$. This expression will be deduced from Theorem \ref{thm Bochner} by comparing the square of $D^E_{p, t}$ to the square of $D_{p, t}$. The main difference between these two involves the \emph{generalised moment map}
\[
\mu^E \in \kg^* \otimes \End(E),
\]
defined by 
\[
2\pi i \mu^E_X = \cL^E_X - \nabla^E_{X^M} \quad \in \End(E),
\]
for all $X \in \kg$, where $\cL^E_X$ is the Lie derivative of sections of $E$ with respect to $X$. Using the metric on $\gstarM$, we obtain 
\[
(\munab, \mu^E) \in \End(E).
\]
\begin{proposition} \label{prop Bochner twist}
On $G$-invariant sections of $\cS_p \otimes E$, one has
\[
(D^E_{p, t})^2 = (D^E_p)^2 + tA\otimes 1_E + (2p + 1)2\pi t \cH^{\nabla} + \frac{t^2}{4}\|{\vnab}\|^2 + 4\pi t 1_{\cS} \otimes (\munab, \mu^E),
\]
with $A \in \End(\cS_p)$ as in Theorem \ref{thm Bochner}.
\end{proposition}

The first step in the proof of Proposition \ref{prop Bochner twist} is a simple relation between the operators $D^E_p$ and $D_p$. Fix a local orthonormal frame $\{e_j\}_{j=1}^{d_M}$ of $TM$. The operator
\[
P_E := \sum_{j=1}^{d_M} c(e_j) \otimes \nabla^{E}_{e_j}
\]
on $\Gamma^{\infty}(\cS_p \otimes E)$
is independent of this frame, and hence globally defined.
\begin{lemma} \label{lem PE}
One has
\[
D^E_p = D_p + P_E.
\]
\end{lemma}
\begin{proof}
In terms of the frame $\{e_j\}_{j=1}^{d_M}$, we have
\[
\begin{split}
D^E_p &= \sum_{j=1}^{d_M} (c(e_j) \otimes 1_E) \bigl(\nabla^{\cS_p}_{e_j} \otimes 1_E + 1_{\cS_p} \otimes \nabla^E_{e_j} \bigr) \\
	&= \sum_{j=1}^{d_M} c(e_j) \nabla^{\cS_p}_{e_j} \otimes 1_E + \sum_{j=1}^{d_M} c(e_j) \otimes \nabla^E_{e_j} \\
	&= D_p + P_E.
\end{split}
\]
\end{proof}

\begin{lemma} \label{lem cv PE}
For all vector fields $v$ on $M$,
\[
(c(v) \otimes 1_E) \circ P_E + P_E \otimes (c(v) \otimes 1_E) = -2 (1_{\cS_p} \otimes \nabla^E_v).
\]
\end{lemma}
\begin{proof}
Since 
\[
c(v)c(e_j) + c(v)c(e_j) = -2(v, e_j)
\]
for all $j$, we see that
\[
\begin{split}
(c(v) \otimes 1_E) \circ P_E + P_E \otimes (c(v) \otimes 1_E) &= \sum_{j=1}^{d_M} \bigl( c(v)c(e_j) + c(e_j)c(v) \bigr) \otimes \nabla^E_{e_j} \\
	&= -2 \sum_{j=1}^{d_M} (v, e_j) 1_{\cS_p}\otimes \nabla^E_{e_j} \\
	&= -2 (1_{\cS_p} \otimes \nabla^E_v).
\end{split}
\]
\end{proof}

Let $(\munab)^*:M \to \kg$ be dual to $\munab$ with respect to a given metric on $\gstarM$. 
For any $G$-equivariant vector bundle $F \to M$, consider the Lie derivative operator $\cL^{F}_{(\munab)^*}$ on $\Gamma^{\infty}(F)$, defined by
\[
(\cL^{F}_{(\munab)^*}s)(m) := (\cL^{F}_{(\munab)^*(m)}s)(m)
\]
for $s \in \Gamma^{\infty}(F)$ and $m \in M$.
\begin{proposition} \label{prop Bochner all}
One has
\begin{multline*}
(D^E_{p, t})^2 = 
(D^E_p)^2 + tA\otimes 1_E + (2p + 1)2\pi t \cH^{\nabla} + \frac{t^2}{4}\|{\vnab}\|^2 + 4\pi t 1_{\cS} \otimes (\munab, \mu^E) \\
-2it \bigl(\cL^{\cS_p}_{(\munab)^*} \otimes 1_E + 1_{\cS_p} \otimes  \cL^{E}_{(\munab)^*} \bigr),
\end{multline*}
with $A \in \End(\cS_p)$ as in Theorem \ref{thm Bochner}.
\end{proposition}
Since for all $X \in \kg$, 
\[
\cL^{\cS_p}_{X} \otimes 1_E + 1_{\cS_p} \otimes  \cL^{E}_{X}
\]
is the Lie derivative on $\cS_p \otimes E$ with respect to $X$, the operator $\cL^{\cS_p}_{(\munab)^*} \otimes 1_E + 1_{\cS_p} \otimes  \cL^{E}_{(\munab)^*}$ equals zero on $G$-invariant sections. Hence Proposition \ref{prop Bochner all} implies Proposition \ref{prop Bochner twist}.

\noindent \emph{Proof of Proposition \ref{prop Bochner all}.}
First note that
\[
\begin{split}
(D^E_{p, t})^2 = (D^E_p)^2 + \bigl( \frac{it}{2}c(\vnab) \otimes 1_E\bigr)^2 
	+ \frac{it}{2} \bigl( D^E_p \circ (c(\vnab) \otimes 1_E) + (c(\vnab) \otimes 1_E) \circ D^E_p\bigr). 
\end{split}
\]
Because of Lemma \ref{lem PE} and \ref{lem cv PE}, we have
\begin{multline*}
D^E_p \circ (c(\vnab) \otimes 1_E) + (c(\vnab) \otimes 1_E) \circ D^E_p = \\
	\bigl(D_p \circ c(\vnab) + c(\vnab) \circ D_p \bigr) \otimes 1_E + (c(\vnab) \otimes 1_E) \circ P_E + P_E \otimes (c(\vnab) \otimes 1_E) = \\
	\bigl(D_p \circ c(\vnab) + c(\vnab) \circ D_p \bigr) \otimes 1_E - 2 (1_{\cS_p} \otimes \nabla^E_{\vnab}).	
\end{multline*}
Furthermore,
\[
\bigl( \frac{it}{2}c(\vnab) \bigr)^2 + \frac{it}{2}\bigl(D_p \circ c(\vnab) + c(\vnab) \circ D_p \bigr) = D_{p, t}^2 - D_p^2.
\]
The right hand side of this equality was computed in Theorem \ref{thm Bochner}. Using the expression obtained there and the above computations, we find that
\[
(D^E_{p, t})^2 = (D^E_p)^2 + tA\otimes 1_E + (2p + 1)2\pi t \cH^{\nabla} + \frac{t^2}{4}\|{\vnab}\|^2  -2it \cL^{\cS_p}_{(\munab)^*}  \otimes 1_E - it 1_{\cS_p} \otimes \nabla^E_{\vnab}.
\]
Now for all $m \in M$,
\[
\nabla^E_{\vnab_m} = 2 \nabla^E_{(\munab)^*(m)^M_m} = 2 \bigl(\cL^E_{(\munab)^*(m)} - 2\pi i \mu^E_{(\munab)^*(m)}(m) \bigr).
\]
Since $\mu^E_{(\munab)^*(m)}(m) = (\munab, \mu^E)(m)$, the claim follows.
\hfill $\square$

\subsection{Localisation}

Proposition \ref{prop loc U}, which is the key step in the proof of Theorem \ref{thm [Q,R]=0 large p}, generalises to twisted Dirac operators in the following way.
\begin{proposition} \label{prop loc twist}
There is a metric on $\gstarM$ such that for every $G$-invariant open neighbourhood $U$ of $\munabinv(0)$, there are $p_E \in \N$ and $t_0, C, b > 0$, such that for all $t \geq t_0$ and $p \geq p_E$, and all $G$-invariant $s \in \Gamma^{\infty}_{tc}(\cS_p \otimes E)^G$ with support disjoint from $U$,
\[
\|f D_{p, t}^Es\|_0^2 \geq C \bigl( \|fs\|_1^2 + (t-b) \|fs\|_0^2\bigr).
\]
Here $\|\cdot \|_k$ denotes the Sobolev norm defined by the operator $D^E_p$, as in \eqref{eq def Sob norm}.
\end{proposition}

Theorem \ref{thm index} follows from Proposition \ref{prop loc twist} in the same way that Theorems \ref{thm quant well defd} and \ref{thm [Q,R]=0 large p} follows from Propositions \ref{prop loc V} and \ref{prop loc U}, as described in Subsection \ref{sec proofs theorems}. The topological expression for the index of $D_{M_0}$ then follows from the Atiyah--Singer index theorem.
We now do not use an analogue of  Proposition \ref{prop loc V} (localisation to neighbourhoods of $\Crit(\vnab)$ for $p=1$), because for twisted Dirac operators we always use large enough powers of $L$. 

It therefore remains to prove Proposition \ref{prop loc twist}.
This proof is based on a generalisation of Proposition \ref{prop choice metric}.
\begin{lemma} \label{lem bound muE}
There is a $G$-invariant metric on $\gstarM$ such that, in addition to the properties in Proposition \ref{prop choice metric}, there is a $C' > 0$ such that the operator $(\munab, \mu^E)$ satisfies the pointwise estimate
\begin{equation} \label{eq bound muE}
1_{\cS_p} \otimes (\munab, \mu^E) \geq -\|\vnab \|^2 -C'
\end{equation}
(for any $p \in \N$).
\end{lemma}
\begin{proof}
As in Section \ref{sec loc est}, choose a relatively cocompact, $G$-invariant neighbourhood $V$ of $\Crit(\vnab)$. Choose a $G$-invariant, positive function $\psi_E \in C^{\infty}(M)^G$ such that, outside $V$,
\[
\|1_{\cS_p} \otimes (\munab, \mu^E) \| \leq \psi_E \|\vnab \|^2.
\]
Fix any $G$-invariant metric $\{ (\relbar, \relbar)_m \}_{m \in M}$ on $\gstarM$. Consider the metric $\{ \psi_E(m)(\relbar, \relbar)_m \}_{m \in M}$ rescaled by $\psi_E$, and let $(\vnab)^{\psi_E} = \psi_E \vnab$ be the vector field associated to this metric. Then, outside $V$,
\[
\|1_{\cS_p} \otimes \psi_E(\munab, \mu^E) \| \leq  \| (\vnab)^{\psi_E} \|^2
\]
Furthermore, the function $\|1_{\cS_p} \otimes \psi_E(\munab, \mu^E) \|$ is $G$-invariant, and hence bounded on $V$. So there is a $C' > 0$ such that, on all of $M$,
\[
\|1_{\cS_p} \otimes \psi_E(\munab, \mu^E) \| \leq  \| (\vnab)^{\psi_E} \|^2 + C'
\]

Let $\psi \in C^{\infty}(M)^G$ be as in the proof of Proposition \ref{prop choice metric}. Choose a positive function $\tilde \psi \in C^{\infty}(M)^G$ such that
\[
\begin{split}
\tilde \psi^{-1} & \leq \min(\psi^{-1}, \psi^{-1}_E); \\
\| d\tilde \psi^{-1} \| &\leq \| d \psi^{-1} \|. 
\end{split}
\]
(This is possible by Lemma C.3 in \cite{HM}.) Then the metric $\{ \tilde \psi(m)(\relbar, \relbar)_m \}_{m \in M}$ has the properties in Proposition \ref{prop choice metric}, and also satisfies \eqref{eq bound muE}.
\end{proof}

Let $f \in C^{\infty}(M)$ be a cutoff function. 
Analogously to \eqref{eq def D tilde}, we define the operator $\widetilde{D}^E_{p, t}$ on $f \Gamma^{\infty}_{tc}(\cS_p \otimes E)^G$ by
\[
\widetilde{D}^E_{p, t} fs = f {D}^E_{p, t} s
\]
for all $s \in \Gamma^{\infty}_{tc}(\cS_p \otimes E)^G$. Corollary \ref{cor est adjoint} now generalises as follows.
\begin{corollary} \label{cor est adjoint twist}
One has
\[
(\widetilde{D}^E_{p, t})^* \widetilde{D}^E_{p, t} = (\widetilde{D}_p^E)^* \widetilde{D}_p^E + tB + (2p+1)2\pi t \cH^{\nabL} + \frac{t^2}{4} \|{\vnab}\|^2,
\]
where $B$ is a vector bundle endomorphism of $\cS_p \otimes E$ for which there is a constant $C > 0$ such that one has the pointwise estimate
\[
B \geq -C \bigl( \|{\vnab}\|^2 +1\bigr).
\]
\end{corollary}
\begin{proof}
As in Lemma 6.8 of \cite{HM}, one has for $s \in \Gamma^{\infty}_{tc}(\cS_p \otimes E)^G$,
\[
(\widetilde{D}^E_{p, t})^* fs = \widetilde{D}^E_{p, t} fs + 2 (c(df) \otimes 1_E) s.
\]
Hence, as in Lemma 6.9 of \cite{HM}, one deduces from Proposition \ref{prop Bochner twist} that for such $s$,
\begin{multline*}
(\widetilde{D}^E_{p, t})^* \widetilde{D}^E_{p, t} fs = \\
(\widetilde{D}_p^E)^* \widetilde{D}_p^E fs + t\left(  A\otimes 1_E + (2p + 1)2\pi t \cH^{\nabla} + 4\pi 1_{\cS} \otimes (\munab, \mu^E)\right) fs \\
+  \frac{t^2}{4}\|{\vnab}\|^2 fs + it (c(df) c(\vnab) \otimes 1_E) s
\end{multline*}
with $A \in \End(\cS_p)$ as in Theorem \ref{thm Bochner}.

Write
\[
B fs := \left(  A\otimes 1_E + 4\pi 1_{\cS} \otimes (\munab, \mu^E)\right) fs + i (c(df) c(\vnab) \otimes 1_E) s,
\]
for $s$ as above.
By Lemma \ref{lem bound muE}, the metric on $\gstarM$ can be chosen such that there is a $C' > 0$ for which
\[
A\otimes 1_E + 4\pi 1_{\cS} \otimes (\munab, \mu^E) \geq -(1+ 4\pi)\|\vnab\|^2 -C'.
\]
By Lemma 6.10 in \cite{HM}, there is a $C'' > 0$ such that for all $s \in \Gamma^{\infty}_{tc}(\cS_p \otimes E)^G$
\[
\Real \bigl( i (c(df) c(\vnab) \otimes 1_E) s, fs \bigr)_0 \geq -C''\bigl( (\|\vnab\|^2 + 1)fs, fs \bigr)_0.
\]
This implies that
\[
B \geq -(C'+ C'' + 1+ 4\pi) (\|\vnab\|^2 + 1).
\]
\end{proof}

The proof of Proposition \ref{prop loc twist} (and hence of Theorem \ref{thm index}) can now be finished as in the proof of Proposition 6.3 in Section 6.4 of \cite{HM}, with Corollary \ref{cor est adjoint twist} playing the role of Proposition 6.7 in \cite{HM}.


\section{Applications and examples} \label{sec appl ex}

Let us mention some applications and examples of Theorems \ref{thm [Q,R]=0 cocpt}, \ref{thm [Q,R]=0 large p} and \ref{thm index}. We will see that Theorem \ref{thm [Q,R]=0 large p} reduces to a $\Spinc$-version of the result in \cite{MZ} in the cocompact case, 
and discuss how to generate examples of Theorems \ref{thm [Q,R]=0 cocpt} and 
\ref{thm index}.
We show how formal degrees of classes in $K_*(C^*_rG)$, generalising formal degrees of discrete series representations, and related to certain charcteristic classes on $M$. Finally, we use the index formula for twisted $\Spinc$-Dirac operators in Theorem \ref{thm index} to draw conclusions about Braverman's analytic index of such operators.

As before, we assume $G$ is unimodular.

\subsection{Generalising Landsman's conjecture to $\Spinc$-manifolds}

As noted in Subsection \ref{sec cocpt zero}, Theorem \ref{thm [Q,R]=0 large p} implies that the main result in \cite{MZ} generalises to the $\Spinc$-setting. 
\begin{corollary}
In the situation of Theorem \ref{thm [Q,R]=0 large p}, suppose that
 \label{lem cocpt special case}
 $M/G$ is \emph{compact}. Then, in the notation of  Subsection \ref{sec cocpt zero},
\[
R_0\bigl( Q^{\Spinc}_G(M)\bigr) = Q(M_0),
\]
for the $\Spinc$-structures on $M$ and a connections on their determinant line bundles for which Theorem \ref{thm [Q,R]=0 large p} holds.
\end{corollary}
\begin{proof}
If $M/G$ is compact, one may take $t_0 = 0$ in Theorem \ref{thm quant well defd}. (By using $V=M$ in Proposition \ref{prop loc V}.) As noted in Remark \ref{rem L2T cpt}, the fact that all smooth sections are transversally $L^2$ in this case implies that
\[
Q^{\Spin^c}(M)^G = \dim (\ker D^+)^G - \dim(\ker D^-)^G.
\]
Bunke shows in the appendix to \cite{MZ} that this equals $R_0 \bigl( Q^{\Spinc}_G(M)  \bigr)$.
\end{proof}

In other words, an extension of Landsman's conjecture \eqref{eq conj [Q,R]=0 cocpt zero} to the $\Spinc$-case holds for suitable choices of  $\Spinc$-structures  and connections. In fact, Theorem \ref{thm index} can be used to generalise this result to twisted $\Spinc$-Dirac operators.

\subsection{Generating examples} \label{sec ex fibred}

Using the constructions in Subsection \ref{sec fibred prod}, one can generate 
a large class of
examples of Theorems \ref{thm [Q,R]=0 cocpt} and 
\ref{thm index}
from cases where the group acting is compact. Indeed, let $K$ be a compact, connected Lie group, and let $N$ be a manifold equipped with an action by $K$ and a $K$-equivariant $\Spinc$-structure. Let $\munabN: N \to \kk^*$ be the $\Spinc$-momentum map associated to a $K$-invariant Hermitian connection $\nabla^N$ on the determinant line bundle $L^N \to N$ of the $\Spinc$-structure on $N$. Let 
 $v^{\nabla^N}$ be the vector field on $N$ associated to  $\munabN$ as in \eqref{eq def v},  with respect to a single $\Ad^*(K)$-invariant inner product on $\kk^*$. Suppose it has a compact set $\Crit(v^{\nabla^N})$ of zeros. As noted in Lemma 3.24 in \cite{Paradan2}, and on page 4 of \cite{Vergne2006}, this is true if $N$ is real-algebraic and $\mu^{\nabla^N}$ is algebraic and proper. (And also, of course, if $N$ is compact.)

 Let $G$ be a connected, unimodular Lie group containing $K$ as a maximal compact subgroup. Suppose the lift $\widetilde{\Ad}$ in \eqref{eq tilde Ad} exists, which is true if one replaces $G$ by a double cover if necessary. We saw in Subsections \ref{sec fibred prod} and \ref{sec slice red} that the manifold $M := G\times_K N$ has a $G$-equivariant $\Spinc$-structure with determinant line bundle $L^M = G\times_K L^N$. Furthermore, by Proposition \ref{prop slice spinc}, \emph{all} $G$-equivariant $\Spinc$-manifolds arise in this way (though possibly not all Riemannian metrics on such manifolds). In Subsection \ref{sec ind nabla}, a connection $\nabla^M$ on $L^M$ was constructed, such that the associated $\Spinc$-momentum map $\munabM$ is given by \eqref{eq mu N M 1}. 
 
 If $N$ is compact and even-dimensional, then Theorem \ref{thm [Q,R]=0 cocpt} applies, and yields a decomposition of $Q_G^{\Spinc}(M)_r \in K_*(C^*_rG)$. If $N$ is possibly noncompact, then Theorem 
 \ref{thm index}
 applies for a suitable metric on $\gstarM$.
\begin{corollary} \label{cor [Q,R]=0 non-cocpt ind}
Suppose the dimension of $M$ is even. Let $E \to M$ be a $G$-equivariant, Hermitian vector bundle, equipped with a $G$-invariant, Hermitian connection. 
If $0 \in \kk^*$ is a regular value of $\mu^{\nabla^N}$, and $K$ acts freely on $(\mu^{\nabla^N})^{-1}(0)$, then there are a metric on $\gstarM$ and a $p_E \in \N$ such that for all $p \geq p_E$,
\[
\ind^G_{L^2_T} D^E_{p, 1} = \int_{M_0} \ch(E_0) e^{\frac{p}{2}c_1(L_0)} \hat A(M_0).
\]
\end{corollary}
\begin{proof}
By Proposition \ref{prop Spinc fibred}, zero is a $\Spinc$-regular value of $\munabM$. 
By \eqref{eq decomp mu Gxi}, $G$ acts freely on $\munabMinv(0)$.
To apply Theorem 
\ref{thm index}, it
therefore only remains to show that the vector field $v^{\nabla^M}$ on $M$, induced by the momentum map $\mu^{\nabla^M}$ as in \eqref{eq def v}, has a cocompact set $\Crit(v^{\nabla^M})$ of zeros. This follows from the fact that
\[
\Crit(v^{\nabla^M}) = G\times_K \Crit(v^{\nabla^N}),
\]
for a suitable metric on $\gstarM$. This is proved in Lemma \ref{lem crit vM} below. Therefore, Theorem \ref{thm index} implies the claim.
\end{proof} 
\begin{remark}
In the setting of Corollary \ref{cor [Q,R]=0 non-cocpt ind},
Proposition \ref{prop PNxi PMxi} implies that
\[
Q^{\Spin^c}(M)^G = Q^{\Spinc}(M_0) = Q^{\Spinc}(N_0) = Q^{\Spin^c}(N)^K.
\]
\end{remark}

\begin{lemma} \label{lem crit vM}
There is a $G$-invariant metric on $\gstarM$ such that
the set of zeros of the vector field $v^{\nabla^M}$ on $M$, used in the proof of Corollary \ref{cor [Q,R]=0 non-cocpt ind}, equals
\begin{equation} \label{eq crit ind}
\Crit(v^{\nabla^M}) = G\times_K \Crit(v^{\nabla^N}).
\end{equation}
\end{lemma}
\begin{proof}
Let $(\relbar, \relbar)_K$ be an $\Ad^*(K)$-invariant inner product on $\kg^*$ that extends the inner product on $\kk^*$ used to define $v^{\nabla^N}$. Consider the $G$-invariant metric on $\gstarM$ defined by
\[
(\xi, \xi')_{[g, n]} := \bigl(\Ad^*(g)^{-1}\xi, \Ad^*(g)^{-1}\xi'\bigr)_K,
\]
for $\xi, \xi' \in \kg^*$, $g \in G$ and $n \in N$. Let $v^{\nabla^M}$ be defined via this metric. We will show that $v^{\nabla^M}|_N = v^{\nabla^N}$, where we embed $N$ into $M$ via the map $n \mapsto [e, n]$. Then \eqref{eq crit ind} follows by $G$-invariance of both sides.


The dual map $(\mu^{\nabla^M})^*: M \to \kg$, defined with respect to the above metric on $\gstarM$, satisfies
\[
(\mu^{\nabla^M})^*[e, n] = (\mu^{\nabla^N})^*(n),
\]
for all $n \in N$, where $(\mu^{\nabla^N})^*$ is the map dual to $\mu^{\nabla^N}$ with respect to the restriction of $(\relbar, \relbar)_K$ to $\kk^*$. Here $\kk^*$ is embedded into $\kg^*$ via the inner product $(\relbar, \relbar)_K$ (i.e.\ $\kp \subset \kg$ is defined as the orthogonal complement to $\kk$ with respect to the induced inner product on $\kg$). 
%
Hence
\[
\begin{split}
v^{\nabla^M}_{[e, n]} 
	&= 2 \bigl( (\mu^{\nabla^N})^*(n) \bigr)^M_{[e,n]} \\
	&= 2\ddt \left[ \exp \bigl( t(\mu^{\nabla^N})^*(n) \bigr), n \right] \\
	&= 2\ddt \left[ e, \exp \bigl( t(\mu^{\nabla^N})^*(n) \bigr)n \right] \\
	&= v^{\nabla^N}_n, 
\end{split}
\]
so the claim follows.
\end{proof}

\subsection{Characteristic classes and formal degrees}

Theorem \ref{thm [Q,R]=0 cocpt}
is stated in terms of $K$-theory of $C^*$-algebras, but it has purely geometric consequences. In particular, it yields an expression for the formal degrees of discrete series representations of semisimple groups in trems of characteristic classes on coadjoint orbits.


Let $\tau: C^*_rG \to \C$ be the von Neumann trace,  determined by
 \[
 \tau\bigl(R(\varphi)^*R(\varphi) \bigr)=\int_G |\varphi(g)|^2 dg,
 \]
 for $\varphi \in L^1(G) \cap L^2(G)$, where $R$ denotes the right regular representation. This induces a morphism $\tau_* : K_{*}(C^*_{r}G) \to \R$. Wang showed in Proposition 4.4 and Theorem 6.12 in \cite{Wang} that
 \[
 \tau_* \bigl( Q^{\Spinc}_G(M)_r \bigr) = \int_M f e^{\frac{1}{2}c_1(L)} \hat A(M).
 \]
 Here $f$ is a cutoff function as in Definition \ref{def cutoff fn}. 
For $\lambda \in \Lambda_+ + \rho_K$, let $[\lambda] \in K_d(C^*_rG)$ be as in \eqref{eq class lambda}. (As before, $d$ is the dimension of $G/K$.) We define the \emph{formal degree} of $[\lambda]$ as
\[
d_{\lambda} := \tau_*[\lambda] \quad \in \R.
\]
Theorem \ref{thm [Q,R]=0 cocpt} has the following consequence.
\begin{corollary} \label{cor A hat d lambda}
In the setting of Theorem \ref{thm [Q,R]=0 cocpt}, we have
\[
\int_M f  e^{\frac{1}{2}c_1(L^M)} \hat A(M) = \sum_{\lambda \in \Lambda_+ + \rho_K} m_{\lambda} d_{\lambda},
\]
where $m_{\lambda} \in \Z$ is given by the quantisation commutes with reduction relation \eqref{eq [Q,R]=0 cpt}.
\end{corollary}
This corollary is a noncompact generalisation of the equality
\[
\int_N   e^{\frac{1}{2}c_1(L^N)} \hat A(N) = \sum_{\lambda \in \Lambda_+ + \rho_K} m_{\lambda} \dim(V_{\lambda}),
\]
in the compact case. Here $V_{\lambda}$ is the representation space of $\pi^K_{\lambda}$. 

Now suppose $G$ is semisimple with discrete series, and let $\lambda \in \Lambda_+ + \rho_K$. Let $d_{\pi}$ be the formal degree of the discrete series representation $\pi$ with Harish--Chandra parameter $\lambda$.  Then by (5.3) in \cite{HochsPS} and the remarks in Section 2.3 in \cite{Lafforgue}, we have
\[
d_{\pi} = (-1)^{d/2}d_{\lambda}.
\]
(This motivates the term `formal degree' for the number $d_{\lambda}$ in general.) In part (iii) of Proposition 7.3.A in \cite{CM}, Connes and Moscovici gave a decomposition of the $L^2$-index of the $\Spinc$-Dirac operator on a homogeneous space of $G$ into the formal degrees $d_{\pi}$. The left-hand side of the equality in Corollary \ref{cor A hat d lambda} is the $L^2$-index of the $\Spinc$-Dirac operator on $M$ by Theorem 6.12 in \cite{Wang}. Therefore,  Corollary \ref{cor A hat d lambda} is a version of quantisation commutes with reduction for an index as in Connes and Moscovici's result, if $M$ is a homogeneous space.

For specific choices of such homogeneous spaces, one actually only picks up a single formal degree.
Using Proposition 4.4 in \cite{Wang} along with Theorem 6.12 in that paper, or Connes and Moscovici's index theorem, Theorem 5.3 in \cite{CM}, one obtains
\[
d_{\pi} = (-1)^{d/2}\int_{G/K} f \ch(G\times_K V_{\lambda}) \hat A(G/K).
\]
(Also compare this with Corollary 7.3.B in \cite{CM}.)
In a similar way, Corollary 2.8 in \cite{HochsPS} implies that
\[
d_{\pi} =  (-1)^{d/2} \int_{G\cdot \lambda} f e^{\frac{1}{2}c_1(L)} \hat A\bigl(G\cdot \lambda).
\]
Corollary \ref{cor A hat d lambda} is a generalisation of the latter equality from strongly elliptic coadjoint orbits to arbitrary manifolds (satisfying the hypotheses of Theorem \ref{thm [Q,R]=0 cocpt}).


\subsection{Consequences of the index formula for twisted Dirac operators} \label{sec appl ex twist}

The index formula for twisted $\Spinc$-Dirac operators in Theorem \ref{thm index} implies some properties of the index of such operators, which are not a priori clear from Braverman's analytic definition of this index.

Braverman's cobordism invariance result, Theorem 3.6 in \cite{Braverman2}, implies the excision property that the index only depends on data near $\Crit(\vnab)$, as in Lemma 3.12 in \cite{Braverman2}. Because of Theorem \ref{thm index}, the index has a more refined excision property for twisted $\Spinc$-Dirac operators, namely that it only depends on data near the subset $\munabinv(0)$ of $\Crit(\vnab)$.
\begin{corollary}[Excision]
For $j = 1, 2$, let $M_j$ be a $G$-equivariant $\Spinc$-manifold, with spinor bundle $\cS_{M_j} \to M_j$. Let $\nabla^{L_j}$ be a $G$-invariant Hermitian connection on the determinant line bundle $L_j \to M_j$. Let $E_j \to M_j$ be a $G$-equivariant Hermitian vector bundle, equipped with a $G$-invariant Hermitian connection. Suppose these data satisfy the conditions of Theorem \ref{thm index}  for $j = 1, 2$.

In addition, suppose there are $G$-invariant neighbourhoods $U_j$ of $(\mu^{\nabla^{L_j}})^{-1}(0)$, and a $G$-equivariant, isometric diffeomorphism
\[
\varphi: U_1 \to U_2,
\]
such that 
\[
\begin{split}
\varphi \bigl( (\mu^{\nabla^{L_1}})^{-1}(0) \bigr) &= (\mu^{\nabla^{L_2}})^{-1}(0);\\
\varphi^* \bigl(\cS_{M_2}|_{U_2} \bigr)&= \cS_{M_1}|_{U_1};\\
\varphi^*(\nabla^{L_{2}}|_{U_2})&= \nabla^{L_{1}}|_{U_1};\\
\varphi^*(E_2|_{U_2})&= E_1|_{U_1}.
\end{split}
\]
Then there are $G$-invariant metrics on $\kg^*_{M_1}$ and $\kg^*_{M_2}$, and there is a $p_0 \in \N$ such that
for all $p \geq p_0$,
\begin{equation} \label{eq excision}
\ind^G_{L^2_T} D^{E_1}_{p, 1} = \ind^G_{L^2_T} D^{E_2}_{p, 1}.
\end{equation}
\end{corollary}
\begin{proof}
Under the conditions stated, one has 
\[
\begin{split}
(\mu^{\nabla^{L_1}})^{-1}(0)/G &\cong (\mu^{\nabla^{L_2}})^{-1}(0)/G =: M_0;\\
\bigl(L_1|_{(\mu^{\nabla^{L_1}})^{-1}(0)} \bigr)/G &\cong \bigl(L_2|_{(\mu^{\nabla^{L_2}})^{-1}(0)} \bigr)/G =: L_0;\\
\bigl(E_1|_{(\mu^{\nabla^{L_1}})^{-1}(0)} \bigr)/G &\cong \bigl(E_2|_{(\mu^{\nabla^{L_2}})^{-1}(0)} \bigr)/G =: E_0.
\end{split}
\]
Furthermore, because $\cS_{M_1}$ and $\cS_{M_2}$  coincide on a neighbourhood of $(\mu^{\nabla^{L_1}})^{-1}(0) = (\mu^{\nabla^{L_2}})^{-1}(0)$, the $\Spinc$-structures on $M_0$ defined by these spinor bundles are equal.
Hence by Theorem \ref{thm index}, if $p \geq \max(p_{E_1}, p_{E_2})$, both sides of \eqref{eq excision} equal
\[
\int_{M_0} \ch(E_0) e^{\frac{p}{2} c_1(L_0)} \hat A(M_0).
\]
\end{proof}


A direct consequence of Theorem  \ref{thm index} is that $\ind^G_{L^2_T} D^E_{p, 1} $, when defined, depends polynomially on $p$.
\begin{corollary}
In the setting of Theorem \ref{thm index}, there is a $p_E \in \N$ such that for all $p \geq p_E$, 
\[
\ind^G_{L^2_T} D^E_{p, 1} = \sum_{k = 0}^{(\dim M_0)/2} a_k p^k,
\]
with rational coefficients
\[
a_k := \frac{1}{2^k k!}\int_{M_0} \ch(E_0) c_1(L_0)^k \hat A(M_0).
\]
In particular,
\[
\ind^G_{L^2_T} D^E_{p, 1} - \sum_{k = 1}^{(\dim M_0)/2} a_k p^k = \int_{M_0} \ch(E_0) \hat A(M_0)
\]
is independent of $p$.
\end{corollary}

Finally, certain topological invariants of $M_0$ can be recovered as indices on $M$. We illustrate this for a twisted version of the signature.

Let $\gamma$ be the involution of $\bigwedge T^*M \otimes \C$ equal to 
\[
\gamma := i^{(\dim M + {j(j-1)})/{2}} *
\]
on ${\bigwedge}^j T^*M \otimes \C$, where $*$ is the Hodge operator. Consider the de Rham operator
\[
B := d + d^*
\]
on $\Gamma^{\infty}(\bigwedge T^*M \otimes \C)$. It satisfies $B \gamma = -\gamma B$, and hence defines the \emph{signature operator}
\[
B: \Gamma^{\infty}\bigl({\bigwedge}^+ T^*M \otimes \C \bigr) \to \Gamma^{\infty}\bigl({\bigwedge}^- T^*M \otimes \C \bigr)
\]
where the $+$ and $-$ signs denote the $+1$ and $-1$ eigenspaces of $\gamma$. (See e.g.\ Example 6.2 in \cite{LM}.)

For any integer $p$, let $B^{L^p}$ be the signature operator $B$, twisted by $L^p$ via the given connection on $L$. Write
\[
B^{L_p}_{\vnab} := B^{L^p} + \frac{i}{2}c(\vnab).
\]
Let $\cN \to \munabinv(0)$ be the normal bundle to $q^*TM_0$ in $TM|_{\munabinv(0)}$. If $0$ is a $\Spinc$-regular value of $\munab$, then $\cN$ has a $G$-equivariant $\Spin$-structure, with spinor bundle $\cS^{\cN} \to \munabinv(0)$. Let $\cS^{\cN}_0 = (\cS^{\cN}|_{\munabinv(0)})/G \to M_0$ be the induced vector bundle over $M_0$. Then Theorem \ref{thm index} implies  a version of Hirzebruch's signature theorem in this setting.
\begin{corollary}[Twisted signature theorem]
Suppose that $0$ is a $\Spinc$-regular value of $\munab$, and that $G$ acts freely on $(\munab)^{-1}(0)$. Then there is a $G$-invariant metric on $\gstarM$ such that for large enough 
integers $p$,
\[
\ind^G_{L^2_T} B^{L^{p}}_{\vnab} =  \int_{M_0} \ch(\cS^{\cN}_0) e^{(p - \frac{1}{2})c_1(L_0)} L(M_0).
\]
Here $L(M_0)$ is the $L$-class of $M_0$.
\end{corollary}
\begin{proof}
For all $\Spinc$-manifolds $U$, with spinor bundles $\cS_U \to U$, one has
\[
\bigwedge T^*U \cong \Cl(TU) \cong \End(\cS_U) \cong \cS_U \otimes \cS_U^* \cong  \cS_U \otimes \cS_U.
\]
If $U$ is $\Spin$, then under this identification, the signature operator $B_U$ equals the $\Spin$-Dirac operator twisted by $\cS_U$:
\[
B_U = D_U^{\cS_U}.
\]
(See e.g.\ below Proposition 3.62 in \cite{BGV}.) In our setting, $M$ is only $\Spinc$. But as in \eqref{eq decomp Sp}, we have on small enough open sets $U \subset M$,
\[
\cS_p|_U = \cS_U^{\Spin} \otimes L|_U^{p/2},
\]
where $\cS_U^{\Spin}$ is the spinor bundle of a local $\Spin$-structure. Hence, locally, we have for all $p \in \N$,
\[
D_p|_U = (D_U^{\Spin})^{L|_U^{p/2}},
\]
the local $\Spin$-Dirac operator  $D_U^{\Spin}$ coupled to $L|_U^{p/2}$ via the given connection. Twisting $D_p$ by $\cS$, we therefore obtain
\[
D_p^{\cS}|_U = (D_U^{\Spin})^{\cS|_U \otimes L|_U^{p/2}} =  (D_U^{\Spin})^{\cS_U^{\Spin} \otimes L|_U^{(p+1)/2}} = (B|_U)^{L|_U^{(p+1)/2}}.
\]

If $p+1$ is even, then $B^{L^{(p+1)/2}}$ is defined globally, and by the above local argument, it equals $D_p^{\cS}$. This means that for all $k \in \N$, in the notation of Definition \ref{def twisted Dirac},
\[
B^{L^k}_{\vnab} = D_{2k - 1, 1}^{\cS}.
\]
Under the conditions stated, Theorem \ref{thm index} therefore yields the equality
\[
\ind^G_{L^2_T} B^{L^{k}}_{\vnab} =  \int_{M_0} \ch(\cS_0) e^{(k - \frac{1}{2})c_1(L_0)} \hat A(M_0),
\]
for $k$ large enough. Since $\cS_0 = \cS_{M_0} \otimes \cS^{\cN}_0$ and $\ch(\cS_{M_0}) \hat A(M_0) = L(M_0)$, the claim follows.
\end{proof}


\end{document}